\documentclass[11pt]{amsart}

\usepackage[hmargin=0.8in,height=8.6in]{geometry}
\usepackage{amssymb,amsthm, times}
\usepackage{delarray,verbatim}
\usepackage{mathrsfs}
\usepackage{ifpdf}
\ifpdf
\usepackage[pdftex]{graphicx}
 \else
\usepackage[dvips]{graphicx}
 \fi

\usepackage{bm}

\usepackage{ifpdf}
\usepackage{color}
\definecolor{webgreen}{rgb}{0,.5,0}
\definecolor{webbrown}{rgb}{.8,0,0}
\definecolor{emphcolor}{rgb}{0.5,0.95,0.95}

\usepackage{hyperref}
\hypersetup{%
          colorlinks=true,
          linkcolor=webbrown,
          filecolor=webbrown,
          citecolor=webgreen,
          breaklinks=true}
\ifpdf \hypersetup{pdftex,
            pdfstartview=FitH, 
            bookmarksopen=true,
            bookmarksnumbered=true
} \else \hypersetup{dvips} \fi

\newcommand {\ud}{{\rm d}}

\numberwithin{equation}{section}

\newtheorem{assumption}{Assumption}[section]

\newtheorem{theorem}{Theorem}[section]
\newtheorem{proposition}{Proposition}[section]

\newtheorem{remark}{Remark}[section]
\newtheorem{lemma}{Lemma}[section]

\numberwithin{remark}{section} \numberwithin{proposition}{section}
\numberwithin{corollary}{section}
\newcommand {\R}{\mathbb{R}}

\newcommand {\N}{\mathbb{N}}
\newcommand {\p}{\mathbb{P}}

\newcommand {\E}{\mathbb{E}}

\newcommand{\blue}{\textcolor[rgb]{0.00,0.00,1.00}}

\newcommand{\e}{\mathbb{E}}

\begin{document}
 \title[The relative frequency between two CBI\MakeLowercase{s} and their genealogy]{The relative frequency between two Continuous-State Branching Processes with immigration and their genealogy}
\author{Mar\'ia Emilia Caballero, Adri\'an Gonz\'alez Casanova, Jos\'e Luis P\'erez}
\maketitle

\vspace{-.2in}

\begin{abstract}
When  two (possibly different in distribution) continuous-state branching processes with immigration are present, we study the relative frequency of one of them when the total mass is forced to be constant at a dense set of times. This leads to a SDE  whose unique strong solution will be the definition of a $\Lambda$-asymmetric frequency process ($\Lambda$-AFP). We prove that it is a Feller process and  we  calculate a large population limit when the total mass tends to infinity. This allows us to  study the fluctuations of the process around its deterministic limit. Furthermore, we find conditions for the $\Lambda$-AFP to have a moment dual. The dual can be interpreted in terms of selection, (coordinated) mutation, pairwise branching (efficiency), coalescence, and a novel component that comes from the asymmetry between the reproduction mechanisms. In the particular case of a pair of equally distributed continuous-state branching processes the associated $\Lambda$-AFP  will be the dual  of a $\Lambda$-coalescent. The map that sends each continuous-state branching process to its associated $\Lambda$-coalescent (according to the former procedure) is a homeomorphism between metric spaces. 
\end{abstract}

{\small {{\it AMS 2020 subject classifications}:  60J90, 60J80, 92D15, 92D25}

}

\section{Introduction}

Heuristically, there is a strong relation between continuous-state branching (CB) processes and coalescents. However, explicitly providing this relationship is a problem that has a long history. In this paper, among other results, we construct an explicit homeomorphism between the two families of processes. This homeomorphism is constructed using the tools developed in the first part of the paper.
It is well known that  the moment dual of the frequency process is the block counting process of  a $\Lambda$-coalescent (see \cite{Pitman} and \cite{Sagitov}), which can be interpreted as the genealogy of the CB process. Our  homeomorphism sends each CB to its genealogy. 

Our motivation is to study  the dynamics of the genetic profile of  a population consisting of two types of individuals who might reproduce using radically different mechanisms. Each type of individual is modeled by a continuous-state branching processes with immigration (CBI) that are denoted by $X^{(1)}$ and $X^{(2)}$ respectively, which are assumed to be independent. For example, one type can reproduce by seldom big reproduction events while the other reproduces more often but only a few new offsprings are produced at each reproduction event. 
This study will give rise to a new frequency process, which  will (for obvious reasons) be asymmetric.

To this end we will characterize the evolution of the total size of the population and the frequency process associated to one of the two types. Consider 
the process $Z=\{Z_t:t\geq0\}$, describing the total mass of the population, defined by
\begin{align}\label{tpsp}
	Z_t&=X^{(1)}_t+X^{(2)}_t
	, \qquad\text{$t\geq0$},\qquad Z_0=z,
\end{align}
where $z:=x^{(1)}+x^{(2)}:=X^{(1)}_0+X_0^{(2)}$. 
In addition, by the fact that $X^{(1)}$ and $X^{(2)}$ do not hit $0$ at the same time a.s., we consider the frequency process of type $1$ individuals, $R=\{R_t:t\geq0\}$, given by
\begin{align}\label{fp}
	R_t&=\frac{X^{(1)}_t}{X^{(1)}_t+X^{(2)}_t}1_{\{t\leq  \tau\}}+R_{\tau} 1_{\{t>  \tau\}}
	\qquad\text{$t\geq0$},\qquad
	R_0=r,
\end{align}
where $\tau=\inf\{t\geq0: X^{(1)}X^{(2)}=0\}$ and $r=x^{(1)}/(x^{(1)}+x^{(2)})$.

It is important to note that the process $(R,Z)$ has the Markov property and we will show that it can be characterized as the solution to a martingale problem. 
However, the process $R$ is not Markovian by itself, so it is not an autonomous frequency process. The main difficulty in obtaining a notion of a frequency processes is that if we wish to study only the frequency process $R$ while preserving a Markovian structure, then (in general) the total mass should be constant in time. The class of $\alpha$-stable CB processes is an interesting special case, not only because their associated relative frequency processes become Markovian in its own right after a time-change, but also because the latter constitute the moment duals of the $\beta$-coalescents (see for instance \cite{7authors}). The restriction of a fixed population size is a classic assumption in population genetics but seemed difficult to impose in the present case without losing the properties of the CB processes $X^{(1)}$ and $X^{(2)}$. For example, naively conditioning the processes so that the total mass stays close to some value would inhibit big jumps. To overcome this difficulty, we will consider the dynamics of the frequency process $R$ but at certain points in time we return the process $Z$ to its original value $z$. By taking the lengths of the intervals between these \textit{culling times} tend to zero and speeding up time, we are able to derive as a scaling limit an autonomous Markov process $R^{(z,r)}=\{R^{(z,r)}_t:t\geq0\}$; which we call the \textbf{$\Lambda$-asymmetric frequency process (AFP)}. 

In many cases, our $\Lambda$ asymmetric frequency processes have a moment dual which provides a notion of generalized ancestry in the spirit of the celebrated ancestral selection graph \cite{KN}. This generalization also includes cases such as the duality between mutation and death and  the duality between pairwise branching and efficiency (understood as the competition of traits that consume a different amount of resources in order to reproduce) which was described in \cite{GMP}.
When there is no immigration and $X^{(1)}$ and $X^{(2)}$ have the same branching mechanisms, a relationship between the CB processes and $\Lambda$-coalescents is uncovered. This provides a homeomorphism between the metric spaces defined by these two classes of processes. \\


In the last few decades, the relationship between CB processes and coalescents has been a subject to great interest, and several important contributions in this area have been published. We will briefly describe those that inspired this paper. First, Etheridge and March in \cite{EM} and Perkins in \cite{Per} realized that the Fleming-Viot superprocess \cite{FlemingViot},  can be obtained as a functional of two independent Dawson-Watanabe superprocesses (see \cite{DawsonWatanabe} for an introduction to this class of processes). The two-type Fleming Viot process  is the Wright-Fisher diffusion, which is moment dual to the block counting process of the Kingman coalescent.
A similar result was found by Bertoin and Le Gall, who observed that the Bolthausen-Sznitman coalescent describes the genealogy of Neveu's CB process \cite{BLG}. Finally, in a celebrated seven-authors paper  \cite{7authors}, the method of considering the relative frequency between two independent identically distributed $\alpha$-stable CB processes and time-changing it by using a functional of their total mass, reached its highest point.

Inspired in questions posed by Bertoin and Le Gall \cite{BLG}, Berestycki, Berestycki, and Limic \cite{BBL}, based on ideas of Donelly and Kurtz \cite{D-K1,D-K2}, constructed a look-down coupling between $\Lambda$-coalescents and CB processes with the characteristic triplet $(0,0,y^{-2}\Lambda(y))$, where $\Lambda$ is a finite measure in [0,1] that characterizes the $\Lambda$-coalescent. Their construction works for small times and clarifies the relationship between the extinction of a CB process and the coming down from infinity of a $\Lambda$-coalescent. 

Recently Johnston and Lambert \cite{JL} studied the genealogy of general CB processes and discovered that although the genealogy is not a Markovian object in general, it can be coupled to $\Lambda$-coalescents at small times.  
Their idea is to map each $\Lambda$-coalescent to the CB process with the triplet 
$(0,z^{-1}\Lambda\{0\}, y^{-2}(\mathbf{T^{(z)}})^{-1}(\Lambda-\Lambda\{0\}\delta_0))$ (for a precise definition of the triplet characterizing a CB process see Section \ref{Sec_CBI}), where $\mathbf{T^{(z)}}:\mathcal{M}[0,\infty)\mapsto \mathcal{M}[0,1]$ ($\mathcal{M}[0,\infty)$ and $\mathcal{M}[0,1]$ denoting the space of measures in $[0,\infty)$ and $[0,1]$ respectively)
be such that for every measurable set $A\subset [0,1]$ and $\nu\in \mathcal{M}[0,\infty)$, $\mathbf{T^{(z)}}(\nu)(A)=\nu(T_z^{-1}(A))$ with  $T_z:[0,\infty)\mapsto [0,1]$ given by $T_z(w)=w/(w+z)$. The transformation $T_z$ is very useful and it also plays a central role in the present paper. For example, we will map each $\Lambda$ coalescent to $(0,z^{-1}\Lambda\{0\}, zy^{-2}(\mathbf{T^{(z)}})^{-1}(\Lambda-\Lambda\{0\}\delta_0))$ 
and prove that this mapping is a homeomorphism between the metric spaces of $\Lambda$-coalescents and a particular subset of the class of CB processes. Note that our map differs slightly from the map obtained in \cite{JL} because we use the map that arises naturally from the duality properties of the $\Lambda$-asymmetric frequency process that we will introduce.

In two innovative (although not very well-known) papers Gillespie \cite{Gill73, Gill74}, 
introduces a stochastic differential equation (SDE) to study the probability of fixation of an allele in a scenario where two types compete. He assumes
that both types reproduce for some time according to a Feller's branching diffusion with possibly different parameters, and then some external force (such as winter) discards individuals randomly to maintain a constant population size at each sampling time. 
His ideas reinforced our methods and (as far as we are aware) Gillespie found the first known relationship between Feller's branching diffusions and frequency processes.\\
Gillespie introduced the mathematical idea of culling, which can be easily understood in biological terms if of one thinks in the Long Term Experiment with E. Coli, known as the Lenski experiment. The experiment in the Lenski Lab consists in placing bacteria in fresh medium, the bacteria eats the nutrients in the medium and reproduces. On the next day  a sample is taken from the population, which by that time has growth from around $5\times10^{6}$ cells to $5\times10^8$ cells, and inoculate fresh medium to restart the cycle (see \cite{GKWY} for a mathematical introduction to the experiment). The idea of constructing coalescence processes, for example the $\beta$-coalescent in the case of \cite{Schweinsberg}, by means of pruning Galton-Watson processes had been welcomed in the scientific community since at least few years after the introduction of $\Lambda$-coalescents. These results have the culling/pruning idea in common with our approach. However, they have the advantage (from the point of view of applications to biology) of being related to an individual based model. 
As Galton Watson processes can be rescaled to CB processes, we believe that a diagonally limit can lead to an individual based model for the $\Lambda$-asymmetric frequency process. In order to use our model to study a real experiment, the first step would be to start with a discrete model in which sampling occurs in meaningful units of time, for example days, and obtaining a transparent and simultaneous scaling of space and sampling times, that converges to a process belonging to the family introduced in this paper.

We summarise the main results and give an outline of the paper:
\begin{itemize}
	\item[(1)] The  construction  and the study of $\Lambda$-asymmetric frequency processes. This is the core of our paper and it is done in Sections \ref{two_dim} and \ref{proc_r}.
	We first prove that the two-dimensional process $(R, Z)$ satisfies a martingale problem (Section  \ref{two_dim}). In Section  \ref{proc_r} we introduce the $\Lambda$-asymmetric frequency process $R^{(z,r)}$ as the solution of an SDE. We prove that there exists a unique strong solution to this SDE and show that the solution is a Feller process. The culling procedure, discussed above, is a crucial ingredient of the construction of the $\Lambda$-asymmetric frequency process and it is defined in Section \ref{NC}, where we also prove that the scaling limit of the sequence of processes obtained by the culling procedure corresponds to the process $R^{(z,r)}$.
	\item[(2)] In Section \ref{fluctuations}, we derive a large population limit for the $\Lambda$-asymmetric frequency process by making the total mass go to infinity. This limit is deterministic and consists of a logistic equation that provides a natural notion of the Malthusian in the context of competing populations with different branching mechanisms. Furthermore, we quantify the error of the deterministic approximation through a fluctuation result. We obtain explicitly the Gaussian process that characterizes the fluctuations of the process $R^{(z,r)}$ around the limiting logistic equation. This result besides being of mathematical interest can trigger research in the direction of statistical tests. 
	\item[(3)]  In Section \ref{duality} we study moment duality for $\Lambda$-asymmetric frequency processes. In particular Theorem \ref{theo_dual} gives conditions for the $\Lambda$-asymmetric frequency process $R^{(z,r)}$ to have a moment dual. This gives a notion of generalized ancestry in presence of possibly skewed and asymmetric reproduction mechanisms, population dependent variance, mutation, coordinated mutation, and selection. It is important to note that as a particular case the dual becomes a $\Lambda$-coalescent in the absence of immigration and when the two independent CB processes have the same branching mechanism.
	\item[(4)] Section \ref{CBduality} is dedicated to the case of  two equally distributed independent  CB processes. The former  procedure leads to a homeomorphism  between the space of $\Lambda$-coalescents equipped with  the Skorohod $J_1$-topology and a subspace of the CB processes equipped with the uniform Skorohod topology. This is the content of Theorem  \ref{homeomorphism}. The subspace of the CB processes homeomorphic to the $\Lambda$-coalescents can be thought of as the quotient space that is obtained by using the equivalence relation in which two CB processes are related if and only if they have the same diffusion term $c$ and the same L\'evy measure $\nu$. 
	\item[(5)] To illustrate the main ideas of the paper, we give a simple example in Section \ref{Ex} of the $\Lambda$-asymmetric Eldon-Wakely coalescent. With this example, we also show that the map sending a pair of CB processes to their dual is not continuous if one considers CB processes with different branching mechanisms.
	\item[(6)] In Section \ref{biol_remarks}, we provide some biological remarks related to the evolutionary forces that emerge from the asymmetry between the reproduction and immigration mechanisms of two competing CBI processes, which can be observed from the generator of the moment dual of the $\Lambda$-asymmetric frequency process.
\end{itemize}
Finally, it is important to note that our results bring us back to the discussion of an important biological observation that was made by Gillespie (see \cite{Gill73,Gill74}), which is that the variance of the reproduction mechanisms of competing phenotypes is also under natural selection: mechanisms with low variance are prone to fixation (see also \cite{Taylor} for a more modern discussion on the effect of variance in presence of skewed reproduction mechanisms).  Gillespie introduced an asymmetric version of the Wright-Fisher diffusion, which takes into account the difference in the variance in populations that grow by the action of frequent and small reproduction events. We extend this result to include seldom and big reproduction events, and we construct a reasonable model to study a vast spectrum of questions arising in biology and ecology. We believe that this new family of models can be useful to estimate the probabilities of fixation of competing traits that use radically different reproduction mechanisms. Recursions for the moments of the frequency process can be derived whenever the assumptions for having moment duality are fulfilled. Perhaps in the future, at least in some cases, explicit formulas for fixation probabilities and expected fixation times can be calculated. We believe that many questions arise from the construction of this family of processes: any classic problem in population genetics can be posted for this family of processes, and its answer seems interesting from a mathematical and a biological point of view.

\section {Notations and prerequisites}

To introduce our main results we first need to recall a few facts about CBI's and coalescent processes
\subsection { Continuous-state branching processeses with immigration}\label{Sec_CBI}
Continuous-state branching processeses with immigration $X:=\{X_t:t\geq0\}$ are $[0,\infty]$-valued strong Markov processes that are the continuous-time and state versions of Galton-Watson processes with immigration and were introduced by Kawazu and Watanabe in \cite{KW}, where they show that they are described in terms of a branching mechanism $\psi$ of the form 
\begin{align*}
	\psi(\lambda)&=b\lambda+c\lambda^2 +\int_{(0,\infty)}(e^{-\lambda x}-1+\lambda x1_{(0,1)}(x))m(dx),\qquad\lambda\geq0,
\end{align*}
where $b\in\R$, $c\geq0$, and $m$ is a measure concentrated on $(0,\infty)$ which satisfies that $\int_{(0,\infty)}(1\wedge x^2)m(dx)<\infty$, and a general immigration mechanism given by
\begin{align*}
	\varphi(\lambda)=\eta\lambda +\int_0^{\infty}(1-e^{-\lambda x})\nu(dx),\qquad\lambda\geq0,
\end{align*}
where $\eta\geq0$ and $\nu$ is a measure on $(0,\infty)$ such that $\int_{(0,\infty)}(1\wedge x)\nu(dx)<\infty$. We will write by $\mathbb{P}_{x}$ the law of the process conditioned on the event $\{ X_0=x  \}$ and $\E_x$ as the associated expectation operator.

More precisely, its semigroup is characterized by its Laplace transform as follows
\begin{align*}
	\log \e_x\left[e^{-\lambda X_t}\right]=-xu_t(\lambda)-\int_0^t\varphi(u_{t-s}(\lambda))ds,\qquad t,x,\lambda\geq0.
\end{align*}
where $u_t(\lambda)$ is the unique solution to the following evolution equation
\begin{align*}
	u_t(\lambda)+\int_0^t\psi(u_s(\lambda))ds=\lambda,\qquad t,\lambda\geq0,
\end{align*}
with $u_0(\lambda)=\lambda$.


In the case where there is no immigration (i.e. $\varphi\equiv0$), the triplet $(b,c,m)$ completely characterizes the CB process $X$ and thus it will be referred to as the \textit{characteristic triplet} of $X$.

\subsection{Coalescents and frequency processes}

Given a finite measure $\Lambda$ on $[0,1]$ the block counting process of a $\Lambda$-coalescent, $N=\{N_t:t\geq0\}$, is an $\N$-valued decreasing process that goes from the state $n$ to the state $n-i+1$, for any $i\in\{2,...,n\}$ at rate $\binom{n}{i}\lambda_{n,i}$,  where
\begin{equation*}
	\lambda_{n,i}:=\int_0^1y^{i}(1-y)^{n-i}\frac{\Lambda(dy)}{y^2}.
\end{equation*}
These processes have a biological interpretation: they are related to the genealogy of a population (in a generalized Wright-Fisher model) and they are moment duals to frequency processes, which are the solutions to the following class of SDE's:
\[
R_t=\Lambda(\{0\})\int_0^t\sqrt{R_s(1-R_s)}dB_t+\int_{0}^t\int_{0^+}^1\int_{0}^1y(1_{\{\theta<R_s\}}-R_s)\tilde{N}(ds,dy,d\theta), \qquad t\geq 0,
\]
where $B=\{B_t:t\geq0\}$ is a Brownian motion and $\tilde{N}(ds,dy,d\theta)$ is a compensated Poisson random measure with state-space $\R^+\times [0,1]\times [0,1]$ and intensity measure $dt\times \Lambda(dy)y^{-2}\times d\theta$. The existence and uniqueness of a strong solution to this SDE can be consulted in \cite{DL} and its lookdown construction in \cite{D-K1}. The solutions are called frequency processes because they arise as scaling limits of the frequency of individuals of a certain type in a generalized Wright-Fisher model. In light of these facts, moment duality relates the genealogy of a population with the evolution of its genetic profile. To say that $R$ and $N$ are moment duals is equivalent to saying that for all $x\in[0,1]$, $n\in \N$ and $t>0$
\begin{equation*}
	\E_x[R_t^n]=\E_n[x^{N_t}].
\end{equation*}
Frequency processes are moment duals of block counting processes of coalescent processes. The simplest example is the duality between the Wright-Fisher diffusion and the Kingman coalescent \cite{Kingman}, which was extended in many directions (e.g. to include selection \cite{KN}).  Indeed, many  evolutionary forces rule the fate of populations and the shape of their genealogies; which can be included in a generalized Wright-Fisher model, and can lead to generalizations of coalescents and frequency processes. In many cases, the duality property holds, this is the case of models including mutation, varying population size, geographic structure, latency, etc. However, moment duality does not always work, the most (in)famous example is selection that change signs, which could be a model for a population adapted to winter competing with one adapted to summer.
We refer to \cite{Nat} for further insight in coalescent theory, and to \cite{JK} for an introduction to  moment duality.

\section{Frequency and total size of a two population processes}\label{two_dim}
In this section we will consider two independent CBI's as a model for two subpopulations of different types. We are interested in describing the total size of the population and the associated frequency process. 

To this end, consider two independent CBI's,  $X^{(i)}=\{X^{(i)}_t:t\geq0\}$ $i=1,2$, with branching mechanisms given by
\begin{align}\label{bran_mech}
	\psi^{(i)}(\lambda)&=b^{(i)}\lambda+c^{(i)}\lambda^2 +\int_{(0,\infty)}(e^{-\lambda x}-1+\lambda x1_{(0,1)}(x))m^{(i)}(dx),\qquad\text{ $\lambda\geq0$, $i=1,2$,}
\end{align}
and immigration mechanisms  
\begin{align}\label{imm_mech}
	\varphi^{(i)}(\lambda)=\eta^{(i)}\lambda +\int_0^{\infty}(e^{-\lambda x}-1)\nu^{(i)}(dx),\qquad\text{ $\lambda\geq0$, $i=1,2$.}
\end{align}
For each $i=1,2$, let us consider a standard Brownian motion $B^{(i)}:=\{B^{(i)}_t:t\geq0\}$, a Poisson random measure $N^{(i)}(ds,dz,du)$ on $(0,\infty)^3$ with intensity measure $dsm^{(i)}(dz)du$ and an independent subordinator $\xi^{(i)}=\{\xi^{(i)}_t:t\geq0\}$ with Laplace exponent given by $\varphi^{(i)}$ as in \eqref{imm_mech}. All these elements are assumed to be defined in the same complete probability space and they are independent of each other. It is a well-known fact (see for instance Section 9.5 in \cite{ZL} or Proposition 4 in \cite{CLU} for the case with no immigration)  that for each $i=1,2$, the process $X^{(i)}$ can be seen as the solution to the following stochastic differential equation
\begin{align}\label{CBI_SDE}
	X^{(i)}_t&=x ^{(i)}+\int_0^t\sqrt{2c^{(i)}X^{(i)}_{s-}}dB^{(i)}_s-b^{(i)}\int_0^tX^{(i)}_sds+\int_0^t\int_{[1,\infty]}\int_0^{X^{(i)}_{s-}}zN^{(i)}(ds,dz,du)\notag\\&+\int_0^t\int_{(0,1)}\int_0^{X^{(i)}_{s-}}z\tilde{N}^{(i)}(ds,dz,du)+\xi^{(i)}_t,\qquad t\geq0,
\end{align}
where $\tilde{N}^{(i)}(ds,dz,du)=N^{(i)}(ds,dz,du)-dsm^{(i)}(dz)du$ denotes the compensated associated random measure.

Let us consider the process $(R,Z)$ given by \eqref{tpsp} and \eqref{fp}. It is important to note here that $(R,Z)$ is a Markov process, and (as shown in the next result) it can be characterized as the solution to a martingale problem. We denote the law of the process $(R,Z)$ starting from the initial position $(r,z)\in[0,1]\times(0,\infty)$ by $\mathbb{P}_{(r,z)}$. Accordingly, we write $\E_{(r,z)}$ for the associated expectation operator.

For  $\varepsilon\in(0,z)$ and $L>z$, we denote by $\tau_{\varepsilon}^-:=\inf\{t\geq0: Z_t=\varepsilon\}$ and $\tau_{L}^+=\inf\{t\geq0: Z_t>L\}$, the first hitting time of $\varepsilon$ and the first passage time above the level $L>z$ for the process $Z$, respectively. We note that the fact that the process $Z$ has no negative jumps implies that it cannot drop below the level $\varepsilon$ without hitting it, so the first passage time below the level $\varepsilon>0$ coincides with $\tau_{\varepsilon}^-$ a.s. The two-dimensional process  $(R,Z)$ stopped at the time $\tau = \tau_{\varepsilon}^-\wedge\tau_{L}^+$ describes the dynamics between the two subpopulations that were originally described by the processes $X^{(1)}$ and $X^{(2)}$ before the size of the population becomes relatively small or explodes.

The description of the dynamics of the process $(R,Z)$ until the stopping time $\tau$ as a solution of a martingale problem is provided in the next result and the proof is deferred to Appendix \ref{App_1}.

\begin{proposition}\label{infinitesimal_generator}
	For any $f\in\mathcal{C}^2([0,1]\times[0,\infty))$, the process
	\begin{align*}
		M_t:=f(R_{t\wedge\tau },Z_{t\wedge\tau })-f(r,z)-\int_0^{t\wedge\tau }\mathcal{L}f(R_s,Z_s)ds,
	\end{align*}
	is a local martingale, where 
	\begin{align*}%
		&\mathcal{L}f(r,z)=-b^{(1)}\partial_1f\left(r,z\right)r(1-r)+c^{(1)}\left(\partial_{21}f\left(r,z\right)r(1-r)+\partial_{22}f\left(r,z\right)rz\right)\notag\\
		&-b^{(1)}rz\partial_2f\left(r,z\right)+c^{(1)}\partial_{12}f\left(r,z\right)(1-r)r+\frac{c^{(1)}}{z}\left(\partial_{11}f\left(r,z\right)r(1-r)^2-\partial_1f\left(r,z\right)2r(1-r)\right)\notag\\
		&+rz\int_{(0,\infty)}\Bigg[f\left(r\left(1-\frac{w}{z+w}\right)+\frac{w}{z+w},z+w\right)-f\left(r,z\right)\notag\\&\hspace{ 5cm}-w1_{(0,1)}(w)\left(\partial_1f\left(r,z\right)\frac{(1-r)}{z}+\partial_2f\left(r,z\right)\right)\Bigg]m^{(1)}(dw)\notag\\
		&+\eta^{(1)}\partial_1f\left(r,z\right)\frac{(1-r)}{z}+\eta^{(1)}\partial_2f\left(r,z\right)+c^{(2)}\left(-\partial_{21}f\left(r,z\right)r(1-r)+(1-r)z\partial_{22}f\left(r,z\right)\right)\notag\\&+\int_{(0,\infty)}\left[f\left(r\left(1-\frac{w}{z+w}\right)+\frac{w}{z+w},z+w\right)-f\left(r,z\right)\right]\nu^{(1)}(dw)\notag\\
		&+b^{(2)}\partial_1f\left(r,z\right)r(1-r)-b^{(2)}(1-r)z\partial_2f\left(r,z\right)-c^{(2)}\partial_{12}f\left(r,z\right)r(1-r)\notag\\
		&+\frac{c^{(2)}}{z}\left(\partial_{11}f\left(r,z\right)r^2(1-r)+\partial_1f\left(r,z\right)2r(1-r)\right)\notag\\
		&+(1-r)z\int_{(0,\infty)}\Bigg[f\left(r\left(1-\frac{w}{z+w}\right),z+w\right)-f\left(r,z\right)\notag\\&\hspace{ 5cm}-w1_{(0,1)}(w)\left(-\partial_1f\left(r,z\right)\frac{r}{z}+\partial_2f\left(r,z\right)\right)\Bigg]m^{(2)}(dw)\notag\\
		&-\eta^{(2)}\partial_1f\left(r,z\right)\frac{r}{z}+\eta^{(2)}\partial_2f\left(r,z\right)+\int_{(0,\infty)}\left[f\left(r\left(1-\frac{w}{z+w}\right),z+w\right)-f\left(r,z\right)\right]\nu^{(2)}(dw).
	\end{align*}
\end{proposition}
\section{$\Lambda$-asymmetric frequency processes and culling of population processes}\label{proc_r}
In Section \ref{two_dim} we described the dynamics of the two-dimensional process $(R,Z)$ as the solution of a martingale problem. The process $(R,Z)$ depicts the frequency of one of the types and the total size of the population, respectively. Inspired by many models in population genetics, we would like to have a description of the frequency process under the additional assumption that the total size of the population remains constant in time. To this end we will apply a sampling/culling procedure to the process $(R,Z)$ and obtain a new class of frequency process that we call \textit{$\Lambda$-asymmetric frequency processes}.

We begin this section with a precise definition of $\Lambda$-asymmetric frequency processes.
\subsection{$\Lambda$-asymmetric frequency processes}
Consider, for $i=1,2$, $b^{(i)}\in\R$, $c^{(i)}\geq0$, $\eta^{(i)}\geq 0$, and $m^{(i)}, \nu^{(i)}$  measures on $(0,\infty)$ such that 
$
\int_{(0,\infty)}(1\wedge x^2)\mu^i(dx), \quad \text{and,} \quad \int_{(0,\infty)}(1\wedge x)\nu^{(i)}(dx).
$
For $z>0$ (representing the population size) and $r\in[0,1]$, let us consider the process $R^{(z,r)}=\{R^{(z,r)}_t:t\geq 0\}$ given as the solution to the following stochastic differential equation 
\begin{align}\label{sde_p_cropped}
	dR^{(z,r)}_t&=R^{(z,r)}_{t-}(1-R^{(z,r)}_{t-})1_{\{R^{(z,r)}_{t-}\in[0,1]\}}(b^{(2)}-b^{(1)})dt\notag\\
	&+R^{(z,r)}_{t-}(1-R^{(z,r)}_{t-})1_{\{R^{(z,r)}_{t-}\in[0,1]\}}\Bigg(\frac{2}{z}(c^{(2)}-c^{(1)})\notag\\&+\int_{(0,1)}\frac{w^2}{w+z}m^{(2)}(dw)-\int_{(0,1)}\frac{w^2}{w+z}m^{(1)}(dw)\Bigg)dt\notag\\
	&+\eta^{(1)}\frac{(1-R^{(z,r)}_{t-})}{z}dt-\eta^{(2)}\frac{R^{(z,r)}_{t-}}{z}dt\notag\\&+\sqrt{\frac{2}{z}R^{(z,r)}_{t-}(1-R^{(z,r)}_{t-})[c^{(1)}(1-R^{(z,r)}_{t-})+c^{(2)}R^{(z,r)}_{t-}]}1_{\{R^{(z,r)}_{t-}\in[0,1]\}}dB_t\notag\\
	&+\int_{(0,1)\times(0,\infty)}g^{(z)}(R^{(z,r)}_{t-},w,v)\tilde{N}_1(dt,dw,dv)\notag\\&+\int_{(0,1)\times(0,\infty)}h^{(z)}(R^{(z,r)}_{t-},w,v)\tilde{N}_2(dt,dw,dv)\notag\\
	&+\int_{[1,\infty)\times(0,\infty)}g^{(z)}(R^{(z,r)}_{t-},w,v)N_1(dt,dw,dv)\notag\\&+\int_{[1,\infty)\times(0,\infty)}h^{(z)}(R^{(z,r)}_{t-},w,v)N_2(dt,dw,dv)\notag\\
	&+\int_{(0,\infty)}\tilde{g}^{(z)}(R^{(z,r)}_{t-},w)N_3(dt,dw)+\int_{(0,\infty)}\tilde{h}^{(z)}(R^{(z,r)}_{t-},w)N_4(dt,dw),\notag\\
	R^{(z,r)}_0&=r.
\end{align}
where 
\begin{itemize}
	\item[(i)] $B=\{B_t:t\geq0\}$ is a standard Brownian motion.
	\item[(ii)] $N_1(dt,dw,dv)$ is a Poisson random measure on $(0,\infty)^3$ with intensity measure \linebreak $dtm^{(1)}(dw)dv$. 
	\item[(iii)] $N_2(dt,dw,dv)$ is a Poisson random measure on $(0,\infty)^3$ with intensity measure \linebreak $dtm^{(2)}(dw)dv$. 
	\item[(iv)] $N_3(dt,dw)$ is a Poisson random measure on $(0,\infty)^2$ with intensity measure $dt\nu^{(1)}(dw)$.
	\item[(v)] $N_4(dt,dw)$ is a Poisson random measure on $(0,\infty)^2$ with intensity measure $dt\nu^{(2)}(dw)$.
\end{itemize}
For each $i=1,2$, $\tilde{N}_i(ds,dw,dv)=N_i(ds,dw,dv)-dsm^{(i)}(dw)dv$ denotes the compensated associated random measure.

All of the previous elements are assumed to be defined on the same complete probability space and are independent of each other. Additionally, we have that
\begin{itemize}
	\item[(vi)] For $(x,w,v)\in(0,\infty)^3$ 
	\[
	g^{(z)}(x,w,v):=\frac{w}{z+w}(1-x)1_{\{v\leq xz\}}1_{\{x\in[0,1]\}}.
	\]
	\item[(vii)] For $(x,w,v)\in(0,\infty)^3$ 
	\[
	h^{(z)}(x,w,v):=-\frac{w}{z+w}x1_{\{v\leq (1-x)z\}}1_{\{x\in[0,1]\}}.
	\]
	\item[(viii)] For $(x,w)\in(0,\infty)^2$
	\[
	\tilde{g}^{(z)}(x,w):=\frac{w}{z+w}(1-x)1_{\{x\in[0,1]\}}.
	\]
	\item[(ix)] For $(x,w)\in(0,\infty)^2$
	\[
	\tilde{h}^{(z)}(x,w):=-\frac{w}{z+w}x1_{\{x\in[0,1]\}}.
	\]
\end{itemize}
As we will see later on in this section, the process $R^{(z,r)}$ will be obtained through a sampling/culling procedure of the process $(R,Z)$ to describe the frequency of one of the types in the population under the assumption that the total size of the population is constant and is equal to $z>0$. 
\begin{remark}
	For the diffusion case with no immigration (i.e. $m^{(i)}=\nu^{(i)}=\eta^{(i)}=0$ for $i=1,2$):
	\begin{itemize}
		\item[(i)] The process $R^{(z,r)}$ 
		was obtained by Gillespie in \cite{Gill74} via a sampling procedure, as a continuous-time approximation of a finite gametic-pool selection model. As noted by Gillespie the form of the drift term in \eqref{sde_p_cropped} points to a new form of natural selection acting on the variance. 
		
		\item[(ii)] The same case was also studied by Lambert in \cite{Lambert}, where $R^{(z,r)}$ was obtained by conditioning the process $R$ on the event $\{Z_t=z,\ \text{for all $t\geq0$}\}$, and where the processes $Z$ and $R$  are defined in \eqref{tpsp} and \eqref{fp}, respectively.
	\end{itemize}
\end{remark}

As the first step in our construction, we will show that the process $R^{(z,r)}$ is well-defined, which is given in the next result and the proof is deferred to Appendix \ref{App_2}.
\begin{proposition}\label{exi_uni_sde}
	There exists a unique strong solution $R^{(z,r)}$ to \eqref{sde_p_cropped} such that $R^{(z,r)}_t\in[0,1]$ for all $t\geq0$ $\mathbb{P}$ a.s.
	Furthermore for any $t>0$, there exists a constant $C(t,z)>0$ such that
	\begin{align}\label{estimate}
		\E\left[|R^{(z,r)}_t-R^{(z,\overline{r})}_t|\right]\leq C(t,z)|r-\overline{r}|, \qquad r,\overline{r}\in[0,1].
	\end{align}
\end{proposition}
Throughout the paper we denote the space of finite measures on $[0,1]$ (resp. $[0,\infty)$) by $\mathcal{M}[0,1]$ (resp. $\mathcal{M}[0,\infty)$).
In the next result, we show that the process $R^{(z,r)}$ is Feller and obtain its infinitesimal generator. To this end, we introduce the transformation $\mathbf{T^{(z)}}:\mathcal{M}[0,\infty)\mapsto \mathcal{M}[0,1]$, given by $\mathbf{T^{(z)}}(\nu)(A)=\nu(T_z^{-1}(A))$ for every measurable set $A\subset [0,1]$ and $\nu\in \mathcal{M}[0,\infty)$, with  $T_z:[0,\infty)\mapsto [0,1]$ such that $T_z(w)=w/(w+z)$.
\begin{proposition}\label{feller}
	For any $z>0$, $R^{(z,r)}$ is a Feller process and its infinitesimal generator is given for any $f\in\mathcal{C}^2([0,1])$ by
	\begin{align}\label{inf_gen_cropped}
		&\mathcal{L}^{(z)}f(r)=f'(r)\left[r(1-r)(b^{(2)}-b^{(1)})+\frac{2r(1-r)}{z}\left(c^{(2)}-c^{(1)}\right)\right]+f''(r)\frac{r(1-r)}{z}(c^{(1)}(1-r)+c^{(2)} r)\notag\\
		&+\frac{\eta^{(1)}}{z}f'\left(r\right)(1-r)+\int_{(0,1)}\left[f\left(r(1-u)+u\right)-f\left(r\right)\right]\mathbf{T^{(z)}}(\nu^{(1)})(du)\notag\\
		&+zr\int_{(0,1)}\Bigg[f\left(r(1-u)+u\right)-f\left(r\right)-\frac{u}{1-u}f'\left(r\right)(1-r)1_{(0,1/(1+z))}(u)\Bigg]\mathbf{T^{(z)}}(m^{(1)})(du)\notag\\
		&+z(1-r)\int_{(0,1)}\Bigg[f\left(r(1-u)\right)-f\left(r\right)+\frac{u}{1-u}f'\left(r\right)r1_{(0,1/(1+z))}(u)\Bigg]\mathbf{T^{(z)}}(m^{(2)})(du)\notag\\
		&-\frac{\eta^{(2)}}{z}f'\left(r\right)r+\int_{(0,1)}\left[f\left(r(1-u)\right)-f\left(r\right)\right]\mathbf{T^{(z)}}(\nu^{(2)})(du).
	\end{align}
\end{proposition}
\begin{proof}
	(i) By Theorem 6.4.5 in \cite{A} we have that $R^{(z,r)}$ is a Markov process. Let us consider the semigroup $(\mathcal{T}_t)_{t\geq0}$ of the process $R^{(z,r)}$ given for on any $f\in\mathcal{C}([0,1])$ by
	$
	\mathcal{T}_tf(r)=\E\left[f(R^{(z,r)}_t)\right].
	$
	For $f\in\mathcal{C}^1([0,1])$ and $r,\overline{r}\in[0,1]$, we obtain using \eqref{estimate} 
	\begin{align*}
		\left|\mathcal{T}_tf(r)-\mathcal{T}_tf(\overline{r})\right|\leq \E\left[\left|f(R^{(z,r)}_t)-f(R^{(z,\overline{r})}_t)\right|\right]&\leq \|f'\|_{\infty}\E\left[\left|R^{(z,r)}_t-R^{(z,\overline{r})}_t\right|\right]\notag\\
		&\leq C(t)\|f'\|_{\infty}|r-\overline{r}|,
	\end{align*}
	which implies that the mapping $r\mapsto \mathcal{T}_tf(r)$ is continuous. Meanwhile, for any function $g\in\mathcal{C}_b([0,1])$, we can find a sequence $(f_n)_{n\geq 1}\subset \mathcal{C}^1([0,1])$ such that $f_n\to g$ uniformly on $[0,1]$ as $n\to\infty$. Therefore, $\mathcal{T}_tf_n\to \mathcal{T}_tg$ uniformly on $[0,1]$ as $n\to\infty$. This implies that the mapping $r\mapsto \mathcal{T}_tg(r)$ is continuous, and therefore $\mathcal{T}_t(\mathcal{C}_b([0,1]))\subset\mathcal{C}_b([0,1])$.

	(ii) Fix $f\in\mathcal{C}^2([0,1])$. Using that $R^{(z,r)}$ is a semi-martingale  we can use Meyer-It\^o's formula (cf.\ Theorems II.31 and II.32 of \cite{protter}) to obtain for $t\geq0$
	\begin{align}\label{gen_r_z_1}
		f(R^{(z,r)}_{t})=f(r)+\int_0^{t}\left[C^{(1)}(R^{(z,r)}_s)+C^{(2)}(R^{(z,r)}_s)+C^{(3)}(R^{(z,r)}_s)\right]ds+M_{t},
	\end{align}
	where for $x\in[0,1]$
	\begin{align*}
		&C^{(1)}(x)=f'(x)\left[x(1-x)(b^{(2)}-b^{(1)})+\frac{2x(1-x)}{z}\left(c^{(2)}-c^{(1)}\right)+\frac{\eta^{(1)}}{z}(1-x)-\frac{\eta^{(2)}}{z}x\right]\notag\\&+f''(x)\frac{x(1-x)}{z}\left[c^{(1)}(1-x)+c^{(2)} x\right],\notag\\
		&C^{(2)}(x)=zx\int_{(0,\infty)}\Bigg[f\left(x+\frac{w}{w+z}(1-x)\right)-f\left(x\right)-wf'\left(x\right)\frac{(1-x)}{z}1_{(0,1)}(w)\Bigg]m^{(1)}(dw)\notag
		\end{align*}
	\begin{align*}
		&+\int_{(0,\infty)}\left[f\left(x+\frac{w}{w+z}(1-x)\right)-f\left(x\right)\right]\nu^{(1)}(dw)\notag\\
		&=zx\int_{(0,1)}\Bigg[f\left(x(1-u)+u\right)-f\left(x\right)-\frac{u}{1-u}f'\left(x\right)(1-x)1_{(0,1/1+z)}(u)\Bigg]\mathbf{T^{(z)}}(m^{(1)})(du)\notag\\
		&+\int_{(0,1)}\left[f\left(x(1-u)+u\right)-f\left(x\right)\right]\mathbf{T^{(z)}}(\nu^{(1)})(du),
	\end{align*}
	and
	\begin{align*}
		&C^{(3)}(x)=z(1-x)\int_{(0,\infty)}\Bigg[f\left(x-\frac{w}{w+z}x\right)-f\left(x\right)+wf'\left(x\right)\frac{x}{z}1_{(0,1)}(w)\Bigg]m^{(2)}(dw)\notag\\
		&+\int_{(0,\infty)}\left[f\left(x-\frac{w}{w+z}x\right)-f\left(x\right)\right]\nu^{(2)}(dw)\notag\\
		&=z(1-x)\int_{(0,1)}\Bigg[f\left(x(1-u)\right)-f\left(x\right)+\frac{u}{1-u}f'\left(x\right)x1_{(0,1/1+z)}(u)\Bigg]\mathbf{T^{(z)}}(m^{(2)})(du)\notag\\
		&+\int_{(0,1)}\left[f\left(x(1-u)\right)-f\left(x\right)\right]\mathbf{T^{(z)}}(\nu^{(2)})(du),
	\end{align*}
	and $M=\{M_t: t\geq0\}$ is a local martingale. 
	
	Now, the fact that $f\in\mathcal{C}^2([0,1])$ implies that we can find a constant $K>0$ such that
	\begin{align}\label{bound}
		\left|C^{(1)}(x)+C^{(2)}(x)+C^{(3)}(x)\right|\leq K,\qquad\text{$x\in[0,1]$}.
	\end{align}
	Hence, 
	taking expectations in \eqref{gen_r_z_1}, we obtain
	\begin{align}\label{gen_r_z_2} 
		\E\left[f(R^{(z,r)}_t)\right]-f(r)=\E\left[\int_0^{t}\left[C^{(1)}(R^{(z,r)}_s)+C^{(2)}(R^{(z,r)}_s)+C^{(3)}(R^{(z,r)}_s)\right]ds\right],\qquad t\geq 0.
	\end{align}
	Therefore, using \eqref{bound} and \eqref{gen_r_z_2}, we obtain
	\begin{align*}
		\sup_{r\in[0,1]}|\E\left[f(R^{(z,r)}_t)\right]-f(r)|\leq Kt\rightarrow 0,\qquad\text{as $t\to 0$,}
	\end{align*}
	which implies that $R^{(z,r)}$ is a Feller process. 
	
	(iii) Finally, to obtain the infinitesimal generator of $R^{(z,r)}$, we use \eqref{bound} and \eqref{gen_r_z_2} together with dominated convergence to obtain
	\begin{align*}
		\lim_{t\to0}\frac{\E\left[f(R^{(z,r)}_t)\right]-f(r)}{t}=C^{(1)}(r)+C^{(2)}(r)+C^{(3)}(r).
	\end{align*}
	Hence, the result follows from Theorem 1.33 in \cite{Sch}.
\end{proof}

\subsection{Culling of the population process}\label{NC}

In Section \ref{two_dim}, we obtained a two dimensional Markov process $(R,Z)$ that describes the dynamics of two coexisting populations. The first component of this process describes the frequency of a specific type in the population, while the latter provides information on the total population size. 
One of our main interests in this paper is to study the role of natural selection on the within-generation variance in the offspring distribution. 

Inspired by Gillespie's model \cite{Gill74}, we would like to maintain the total size of the population constant, while allowing the frequency process $R$ to evolve randomly; therefore, obtaining a one-dimensional stochastic process. To this end, throughout the rest of this section we will use a 
sampling method to obtain a stochastic model of the frequency of a particular type in the population under the assumption that the total population size is constant.  

Formally speaking, let us consider a fixed population size level $z>0$, and consider a sequence of homogenous Markov jump processes $\{(\overline{R}^{(z,n)}_t)_{t\geq0}:n\geq 1\}$. We denote the law of the process $\overline{R}^{(z,n)}$ by $\mathbf{P}_r$ when it starts at the position $r\in[0,1]$. For each fixed $n\geq 1$, the Markov process $\overline{R}^{(z,n)}$ has jump times $(T^{n}_m)_{m\geq 1}$ given by independent exponential variables with rate $n$, and a transition kernel $\kappa^{(z,n)}$ defined for $y\in[0,1]$, and $A\in\mathcal{B}([0,1])$, by
\begin{equation}\label{trans_probab}
	\kappa^{(z,n)}(y,A)=\mathbf{P}_y(\overline{R}^{(z,n)}_{T^n_1}\in A):=\mathbb{P}_{(y,z)}\left(R_{\frac{1}{n}\wedge\tau}\in A, Z_{\frac{1}{n}\wedge\tau }\in \R_+\right),
\end{equation}
where  $\tau=\tau_{\varepsilon}^-\wedge\tau_L^+$ (with $\varepsilon$ and $L$ fixed) and the process $(R,Z)$ is the one described in Section \ref{two_dim}.

The infinitesimal generator $\overline{\mathcal{L}}^{(z,n)}$ of the process $\overline{R}^{(z,n)}$ is given for any $f\in\mathcal{C}([0,1])$ by
\begin{align}\label{dis_gen}
	\overline{\mathcal{L}}^{(z,n)}f(r)=n\int_{[0,1]}(f(y)-f(r))\kappa^{(z,n)}(r,dy), \qquad\text{$r\in[0,1]$.}
\end{align}

Intuitively, for each fixed $n\geq 1$, we can think $\overline{R}^{(z,n)}_{T^n_1}$ as a sampling of the first coordinate of the process $(R_{t\wedge\tau},Z_{t\wedge\tau})_{t\geq0}$ started at the position $(r,z)$ at time $t=1/n$. Then, using the fact that the process $(R,Z)$ is a homogenous Markov process, we restart the process  $(R_{t\wedge\tau},Z_{t\wedge\tau})_{t\geq0}$ at the initial position $(R_{\frac{1}{n}\wedge\tau},z)$ and we sample the process at time $t=1/n$ to define $\overline{R}^{(z,n)}_{T^n_2}$. By continuing this procedure we obtain the process $\overline{R}^{(z,n)}$.

From the previous construction, we note that to define the process $\overline{R}^{(z,n)}$ in the time interval $[T^{n}_{m-1},T^{n}_{m})$ for each $m=1\dots$, we consider the evolution of the process $(R,Z)$ in the time interval $[0,\frac{1}{n}\wedge\tau)$ and starting from the state $(\overline{R}^{(z,n)}_{T^n_{m-1}},z)$. Hence, the process $\overline{R}^{(z,n)}$ evolves as the first coordinate of the process $(R,Z)$ but the fluctuations of the total size process $Z$ around $z$ become smaller as we take $n\to\infty$ since the times between jumps converge to zero. 

In the next result, we will show that the sequence of Markov jump processes $\{\overline{R}^{(z,n)}:n\geq 1\}$ converges weakly to the process $R^{(z,r)}$ given as the unique solution to \eqref{sde_p_cropped}, which (by construction) can be understood as having the same dynamics of the first coordinate of the process $(R,Z)$ but with total population size constant and equal to $z>0$.


\begin{theorem}\label{theo_cull}
	For any fixed $z>0$ and $T>0$, $\overline{R}^{(z,n)}\rightarrow R^{(z,r)}$ as $n\to\infty$ weakly in $\mathbb{D}([0,T],[0,1])$.
\end{theorem}
\begin{proof}
	Observing that $R^{(z,r)}_t$ is a Feller process and that all the processes involved take values in the compact interval $[0,1]$, and following Theorem 17.28 in \cite{Ka}, we are only left with the task of proving 
	\begin{align}\label{lim_gen}
		\overline{\mathcal{L}}^{(z,n)}f(r)\to\mathcal{L}^{(z)}f(r),\qquad\text{uniformly on $[0,1]$, as $n\to\infty$,}
	\end{align}
	for every $f\in \mathcal{C}^2([0,1])$. To this end, we have by \eqref{trans_probab} together with \eqref{dis_gen}
	\begin{align}\label{disc_gen_2}
		\overline{\mathcal{L}}^{(z,n)}f(r)=n\left[\E_{(r,z)}\left[f(R_{n^{-1}\wedge\tau})\right]-f(r)\right].
	\end{align}
	To compute the limit in \eqref{lim_gen}, we note that Proposition \ref{infinitesimal_generator} 
	implies that
	\begin{align}\label{loc_mart}
		f(R_{n^{-1}\wedge \tau})-f(r)=\int_0^{n^{-1}\wedge \tau}B(R_s,Z_s)ds+M_{n^{-1}\wedge \tau},
	\end{align}
	where 
	\begin{align*}
		&B(r,z):=-b^{(1)}f'\left(r\right)r(1-r)+\frac{c^{(1)}}{z}\left(f''\left(r\right)r(1-r)^2-f'\left(r\right)2r(1-r)\right)\notag\\
		&+b^{(2)}f'\left(r\right)r(1-r)+\frac{c^{(2)}}{z}\left(f''\left(r\right)r^2(1-r)+f'\left(r\right)2r(1-r)\right)\notag\\
		&+rz\int_{(0,\infty)}\Bigg[f\left(r\left(1-\frac{w}{z+w}\right)+\frac{w}{z+w}\right)-f\left(r\right)-w1_{(0,1)}(w)f'\left(r\right)\frac{(1-r)}{z}\Bigg]m^{(1)}(dw)\\&+\int_{(0,\infty)}\left[f\left(r\left(1-\frac{w}{z+w}\right)+\frac{w}{z+w}\right)-f\left(r\right)\right]\nu^{(1)}(dw)+\eta^{(1)}f'\left(r\right)\frac{(1-r)}{z}\notag\\
		&+(1-r)z\int_{(0,\infty)}\Bigg[f\left(r\left(1-\frac{w}{z+w}\right)\right)-f\left(r\right)+w1_{(0,1)}(w)f'\left(r\right)\frac{r}{z}\Bigg]m^{(2)}(dw)\notag\\
		&-\eta^{(2)}f'\left(r\right)\frac{r}{z}+\int_{(0,\infty)}\left[f\left(r\left(1-\frac{w}{z+w}\right)\right)-f\left(r\right)\right]\nu^{(2)}(dw)\Bigg\}ds,
	\end{align*}
	and $M=\{M_t:t\geq 0\}$ is a local martingale.
	
	Using the fact that $f\in \mathcal{C}^2([0,1])$ and that $(R_s,Z_s)$ takes values in $[0,1]\times[\varepsilon,L]$ for $s\in[0,\tau)$ we can find a constant $K>0$ such that
	\begin{align}\label{bound_B}
		|B(R_s,Z_s)|\leq K, \qquad \text{for $s\in[0,\tau)$ $\mathbb{P}$-a.s.}
	\end{align}
	Therefore, by \eqref{loc_mart} and \eqref{bound_B} we obtain that $(M_{t\wedge \tau})_{t\geq0}$ is indeed a true martingale. 
	Next, by taking expectations in \eqref{loc_mart} we obtain
	\begin{align}\label{lim_gen_1}
		n\left[\E_{(r,z)}\left[f\left(R_{n^{-1}\wedge \tau}\right)\right]-f\left(r\right)\right]=\E_{(r,z)}\left[n\int_0^{n^{-1}\wedge \tau}B(R_s,Z_s)ds\right].
	\end{align}
	Meanwhile, by using \eqref{bound_B} we have
	\[
	n\int_0^{n^{-1}\wedge \tau}B(R_s,Z_s)ds\leq K. 
	\]
	Hence, by dominated convergence together with identity \eqref{inf_gen_cropped}
	\begin{align}\label{lim_gen_2}
		\lim_{n\to \infty}\E_{(r,z)}\Bigg[n\int_0^{n^{-1}\wedge \tau}B(R_s,Z_s)ds\Bigg]=B(r,z)=\mathcal{L}^{(z)}f(r).
	\end{align}
	Therefore, using \eqref{disc_gen_2}, \eqref{lim_gen_1} and \eqref{lim_gen_2} 
	\begin{align}\label{lim_gen_3}
		\lim_{n\to\infty}\overline{\mathcal{L}}^{(z,n)}f(r)=\lim_{n\to\infty}n\left[\E_{(r,z)}\left[f(R_{n^{-1}\wedge\tau})\right]-f(r)\right]=\mathcal{L}^{(z)}f(r).
	\end{align}
	Finally, to show the uniform convergence as in \eqref{lim_gen}, we use Theorem 1.33 in \cite{Sch}.
\end{proof}
\section{Large population asymptotics of $\Lambda$-asymmetric frequency processes}\label{fluctuations}
In this section, we first obtain the large population limit of a $\Lambda$-asymmetric frequency process, while in the second part we study the fluctuations of the process around the large population limit obtained in the first part of this section.
\subsection{Large population limit of $\Lambda$-asymmetric frequency processes}
We will study the asymptotic behavior of the $\Lambda$-asymmetric frequency process $R^{(z,r)}$ as the size of the population becomes large. To this end we introduce the deterministic process $R^{(\infty,r)}=\{R^{(\infty,r)}_t:t\geq0\}$ given by 
\begin{align*}
	R^{(\infty,r)}_t=\frac{re^{\left(\psi^{(2)\prime}(0+)-\psi^{(1)\prime}(0+)\right)t}}{(1-r)+re^{\left(\psi^{(2)\prime}(0+)-\psi^{(1)\prime}(0+)\right)t}},\qquad \text{$t\geq 0$.}
\end{align*} 
Here, we recall that in the case of no immigration i.e. $\xi^{(1)}=\xi^{(2)}=0$, then 
\[
\E_{x^{(i)}}\left[X^{(i)}_t\right]=x^{(i)}\exp\left\{-t\psi^{(i)}(0+)\right\}, \qquad t\geq0.
\]
The large population limit of a  $\Lambda$-asymmetric frequency process $R^{(z,r)}$ is given in the next result, where in particular, we show that the limit does not depend on the immigration mechanisms of the associated CBI's, the proof is deferred to Appendix \ref{App_3}.
\begin{theorem}\label{large_pop}
	Fix $T>0$ and assume that $\int_{(1,\infty)}wm^{(i)}(dw)<\infty$ for $i=1,2$. Then
	\begin{align*}
		\lim_{z\to\infty}\E\left[\sup_{t\leq T}|R^{(z,r)}_t-R^{(\infty,r)}_t|^2\right]=0.
	\end{align*}
\end{theorem}

\subsection{Fluctuations of $\Lambda$-asymmetric frequency processes}
In this section, we will characterize the fluctuations of the process $R^{(z,r)}$ around its large population limit $R^{(\infty,r)}$. In order to do so, we will make the following assumption.
\begin{assumption}\label{assum_1}
	We assume that $\int_{(0,\infty)}w^2m^{(i)}(dw)+\int_{[1,\infty)}w\nu^{(i)}(dw)<\infty$ for $i=1,2$.
\end{assumption}
Let $X^{(\infty)}$ be a zero mean Gaussian process with covariance function $C_{X^{(\infty)}}$ given by
\begin{align*}
	C_{X^{(\infty)}}(s,t):=\int_0^{s\wedge t}e^{2U_u}R^{(\infty,r)}_u(1-R^{(\infty,r)}_u)[\sigma^1(1-R^{(\infty,r)}_u)+\sigma^2R^{(\infty,r)}_u]du, \qquad s,t\geq 0,
\end{align*}
where $\sigma^i=2c^{(i)}+\int_{(0,\infty)}w^2m^{(i)}(dw)$, $i=1,2$, and 
\begin{align}\label{fun_U}
	U_t:=\left(\psi^{(2)\prime}(0+)-\psi^{(1)\prime}(0+)\right)\int_0^t(2R^{(\infty,r)}_s-1)ds, \qquad t\geq0.
\end{align}
Note that we can represent $X^{(\infty)}$ as a time-changed Brownian motion, i.e.
\[
X^{(\infty)}\overset{\mathcal{L}}{=}\left\{W_{\int_0^tC_{X^{(\infty)}}(s,s)ds}; t\geq0\right\},
\]
where $W$ is Brownian motion.

We now state the main result in this section.
\begin{theorem}\label{fluc}
	Under Assumption \ref{assum_1} for any fixed $T>0$, $\sqrt{z}(R^{(z,r)}-R^{(\infty,r)})\to e^{-U}X^{(\infty)}$ as $z\to\infty$ weakly in $\mathbb{D}([0,T],\R)$.
\end{theorem}
In order to prove Theorem  \ref{fluc} we first provide some auxiliary results. To this end, let us define  for $t>0$
\begin{align*}
	&dX^{(z)}_t=e^{U_t}\sqrt{2R^{(\infty,r)}_t(1-R^{(\infty,r)}_t)[c^{(1)}(1-R^{(\infty,r)}_t)+c^{(2)}R^{(\infty,r)}_t]}dB_t\notag\\
	&+\sqrt{z}\int_{(0,\infty)^2}e^{U_t}g^{(z)}(R^{(\infty,r)}_t,w,v)\tilde{N}_1(dt,dw,dv)+\sqrt{z}\int_{(0,\infty)^2}e^{U_t}h^{(z)}(R^{(\infty,r)}_t,w,v)\tilde{N}_2(dt,dw,dv).
\end{align*}
We are now ready to prove the following auxiliary result, the proof is deferred to Appendix \ref{App_4}.
\begin{lemma}\label{iden-limit}
	Fix $T>0$ and consider a sequence $(z_n)_{n\geq1}$ such that $\lim_{n\to\infty}z_n=\infty$.
	Assume that $X^{(z_n)}\to Y^{(\infty)}$  as $n\to\infty$ weakly in $\mathbb{D}([0,T], \R)$.
	Then, $Y^{(\infty)}\overset{\mathcal{L}}{=}X^{(\infty)}$ as elements of $\mathbb{D}([0,T], \R)$.
\end{lemma}
Now, let us define for $t\geq0$
\begin{align}\label{def_A_4}
	&A^{(4,z)}_t:=\int_0^te^{U_s}\sqrt{\frac{2}{z}R^{(z,r)}_{s-}(1-R^{(z,r)}_{s-})[c^{(1)}(1-R^{(z,r)}_{s-})+c^{(2)}R^{(z,r)}_{s-}]}dB_s\notag\\
	&+\int_0^t\int_{(0,\infty)^2}e^{U_s}g^{(z)}(R^{(z,r)}_{s-},w,v)\tilde{N}_1(ds,dw,dv)+\int_0^t\int_{(0,\infty)^2}e^{U_s}h^{(z)}(R^{(z,r)}_{s-},w,v)\tilde{N}_2(ds,dw,dv)\notag\\
	&+\int_0^t\int_{(0,\infty)}e^{U_s}\tilde{g}^{(z)}(R^{(z,r)}_{s-},w)\tilde{N}_3(ds,dw)+\int_0^t\int_{(0,\infty)}e^{U_s}\tilde{h}^{(z)}(R^{(z,r)}_{s-},w)\tilde{N}_4(ds,dw).
\end{align}
We now prove the next auxiliary result, the proof is deferred to Appendix \ref{App_5}.
\begin{lemma}\label{conv_A_2}
	For any $T>0$, 
	\begin{align*}
		\lim_{z\to\infty}\E\left[\sup_{t\in[0,T]}\left|X^{(z)}_t-\sqrt{z}A^{(4,z)}_t\right|^2\right]=0,
	\end{align*}
	where $A^{(4,z)}$ is given in \eqref{def_A_4}.
\end{lemma}
Next, we provide a tightness result for the fluctuations of the process $R^{(z,r)}$ when the size of the population becomes large. It is important to note that in the following proof we use references found in Appendix \ref{App_3}.
\begin{proposition}\label{tightness}
	The family $\{\sqrt{z}(R^{(z,r)}-R^{(\infty,r)}):z\geq 1\}$ is tight in the space $\mathbb{D}(\R_+,\R)$.
\end{proposition}
\begin{proof}
	Let $\delta\in(0,1)$ and $\theta\in[0,\delta]$. Let $T>0$ and $(\tau_n)_{n\geq1}$ be a sequence of stopping times such that $0\leq \tau_n<T$. By \eqref{bound_conv_z_0} we have 
	\[
	\sqrt{z}(R^{(z,r)}_t-R^{(\infty,r)}_t)=\sqrt{z}A^{(1,z)}_t+\sqrt{z}A^{(2,z)}_t+\sqrt{z}A^{(3,z)}_t,\qquad t>0,
	\]
	where $A^{(1,z)}$ is a local martingale, and $A^{(2,z)}$ and $A^{(3,z)}$ are bounded variation processes.
	
	(i) We now provide some estimates for the quadratic variation of the local martingale $A^{(1,z)}$. By \eqref{def_A}, we have that the predictable quadratic variation of $\sqrt{z}A^{(1,z)}$ is given for $t\geq0$ by
	\begin{align*}
		\big[\sqrt{z}&A^{(1,z)}\big]_t=\int_0^t2R^{(z,r)}_s(1-R^{(z,r)}_s)[c^{(1)}(1-R^{(z,r)}_s)+c^{(2)}R^{(z,r)}_s]ds\notag\\
		&+\int_0^t\int_{(0,\infty)}\frac{w^2z^2}{(z+w)^2}(1-R^{(z,r)}_s)^2R_s^{(z,r)}m^{(1)}(dw)ds\notag\\&+\int_0^t\int_{(0,\infty)}\frac{w^2z^2}{(z+w)^2}(1-R^{(z,r)}_s)(R_s^{(z,r)})^2m^{(2)}(dw)ds\notag\\
		&+\int_0^t\int_{(0,\infty)}\frac{w^2z}{(z+w)^2}(1-R^{(z,r)}_s)^2\nu^{(1)}(dw)ds+\int_0^t\int_{(0,\infty)}\frac{w^2z}{(z+w)^2}(R^{(z,r)}_s)^2\nu^{(2)}(dw)ds.
	\end{align*}
	Hence,
	\begin{align}\label{tight_1}
		\sup_{z\geq 1}\sup_{\theta\in[0,\delta]}&\E\left[\left|\big[\sqrt{z}A^{(1,z)}\big]_{\tau_n+\theta}-\big[\sqrt{z}A^{(1,z)}\big]_{\tau_n}\right|\right]\notag\\&\leq \sum_{i=1}^2\left(2c^{(i)}+\int_{(0,\infty)}w^2m^{(i)}(dw)+\int_{(0,\infty)}w\nu^{(i)}(dw)\right)\delta.
	\end{align}
	(ii) Next, proceeding as in \eqref{bv_bounds_1} and \eqref{bv_bounds_2}, we have 
	\begin{align}\label{tight_2}
		\sup_{z\geq 1}\sup_{\theta\in[0,\delta]}&\E\left[\left|\sqrt{z}A^{(2,z)}_{\tau_n+\theta}-\sqrt{z}A^{(2,z)}_{\tau_n}\right|\right]\notag\\&\leq \sum_{i=1}^2\Bigg(2c^{(i)}+\int_{(0,1)}w^2m^{(i)}(dw)+\eta^{(i)}+\int_{(0,\infty)}w\nu^{(i)}(dw)\Bigg)\delta.
	\end{align}
	(iii) For the last term we obtain, proceeding as in \eqref{lip_bounds_1} and \eqref{lip_bounds_2}, that
	\begin{align*}
		\sup_{z\geq 1}&\sup_{\theta\in[0,\delta]}\E\left[\left|\sqrt{z}A^{(3,z)}_{\tau_n+\theta}-\sqrt{z}A^{(3,z)}_{\tau_n}\right|\right]\notag\\&\leq 2\sum_{i=1}^2\left(\int_{[1,\infty)}wm^{(i)}(dw)+|b^{(i)}|\right)\delta\sqrt{z}\E\left[\sup_{t\in[0,T'+\delta]}|R^{(z,r)}_u-R^{(\infty,r)}_u|\right]\notag\\&+\delta \sum_{i=1}^2\left(\int_{[1,\infty)}\frac{w^2\sqrt{z}}{w+z}m^{(i)}(dw)\right)\notag\\
		&\leq C_1\delta+C_2\delta\sqrt{z}\E\left[\sup_{t\in[0,T'+\delta]}|R^{(z,r)}_u-R^{(\infty,r)}_u|^2\right]^{1/2},
	\end{align*}
	where $C_1$ and $C_2$ are positive constants that do not depend on $z$. Meanwhile, using \eqref{gron}, we obtain that 
	\begin{align*}
		\sqrt{z}\E\left[\sup_{t\in[0,T'+\delta]}|R^{(z,r)}_u-R^{(\infty,r)}_u|^2\right]^{1/2}\leq C_3e^{C_4(T'+\delta)^2},
	\end{align*}
	where $C_3$ and $C_4$ are positive constants that are not dependent on $z$. Therefore, 
	\begin{align}\label{tight_3}
		\sup_{z\geq 1}\sup_{\theta\in[0,\delta]}\E\left[\left|\sqrt{z}A^{(3,z)}_{\tau_n+\theta}-\sqrt{z}A^{(3,z)}_{\tau_n}\right|\right]\leq \delta C_5,
	\end{align}
	where $C_5>0$ is independent of $z$.
	
	(iv) For fixed $t>0$, by proceeding as in \eqref{tight_2} and \eqref{tight_3}, we have that there exists a constant $C_6(t)>0$ independent of $z$ such that
	\begin{align*}
		\sup_{z\geq 1}\E\left[\left|\sqrt{z}A^{(2,z)}_{t}+\sqrt{z}A^{(3,z)}_{t}\right|\right]\leq C_6(t).
	\end{align*}
	Meanwhile, by a slight modification of \eqref{tight_1}, we obtain
	\begin{align*}
		\sup_{z\geq 1}\E\left[\left|\sqrt{z}A^{(1,z)}_{t}\right|\right]\leq\sup_{z\geq 1}\E\left[\left|\sqrt{z}A^{(1,z)}_{t}\right|^2\right]^{1/2}\leq C_7(t),
	\end{align*}
	where $C_7(t)>0$ is a constant independent of $z$. Next, by Markov's inequality, we obtain for $M>0$
	\begin{align*}
		\sup_{z\geq 1}\mathbb{P}\left(\sqrt{z}(R^{(z,r)}_t-R^{(\infty,r)}_t)>M\right)&\leq \frac{1}{M}\E\left[\left|\sqrt{z}A^{(1,z)}_{t}+\sqrt{z}A^{(2,z)}_{t}+\sqrt{z}A^{(3,z)}_{t}\right|\right]\notag\\&\leq \frac{1}{M}C_8(t)\to 0,\quad  \text{as $M\to\infty$.}
	\end{align*}
	Therefore, for any $t>0$ the random variable $\sqrt{z}(R^{(z,r)}_t-R^{(\infty,r)}_t)$ is tight in $\R$.
	
	(v) The fact that $\sqrt{z}(R^{(z,r)}_t-R^{(\infty,r)}_t)$ is tight for every $t>0$, together with \eqref{tight_1}, \eqref{tight_2}, and \eqref{tight_3}, implies by the Aldous-Rebolledo criterion, see \cite{RE}, that the family $\{\sqrt{z}(R^{(z,r)}-R^{(\infty,r)}):z\geq 1\}$ is tight in the space of cadlag paths from $\R_+$ to $\R$ with the Skorohod topology.

\end{proof}

We now provide the proof of the main result of this section. We remark that in the following proof we use references found in Appendix \ref{App_3}.
\subsubsection{Proof of Theorem \ref{fluc}}
From Proposition \ref{tightness}, the family $\{\sqrt{z}(R^{(z,r)}-R^{(\infty,r)}):z\geq 1\}$ is relatively compact. Hence, there exists a subsequence
$\left \{ z_n \right \} _{n\geq1}$ such that $\{(\sqrt{z_n}(R^{(z_n,r)}-R^{(\infty,r)}))_{t\geq0}:n\geq1\}$
converges weakly to some $(Y^{(\infty)})_{t\geq0}$ in $\mathbb{D}(\mathbb{R}_+,\R)$. Therefore, it is enough to prove that there is a unique limit point for any convergent subsequence.

Using integration by parts together with \eqref{bound_conv_z_0} and \eqref{fun_U} we have for $t>0$
\begin{align*}
	e^{U_t}\sqrt{z_n}(R^{(z_n,r)}_t-R^{(\infty,r)}_t)&=-(\psi^{(2)\prime}(0+)-\psi^{(1)\prime}(0+))\frac{1}{\sqrt{z_n}}\int_0^te^{U_s}\left[\sqrt{z_n}(R^{(z_n,r)}_t-R^{(\infty,r)}_t)\right]^2ds\notag\\&+\sqrt{z_n}A^{(4,z_n)}_t+\sqrt{z_n}A^{(5,z_n)}_t,
\end{align*}
where $A^{(4,z)}$ is given in \eqref{def_A_4} and
\begin{align}\label{def_A_5}
	&A^{(5,z)}_t:=\int_0^te^{U_s}\Bigg[R^{(z,r)}_s(1-R^{(z,r)}_s)\Bigg(\frac{2}{z}(c^{(2)}-c^{(1)})+\int_{(0,1)}\frac{w^2}{w+z}m^{(2)}(dw)-\int_{(0,1)}\frac{w^2}{w+z}m^{(1)}(dw)\Bigg)\Bigg]ds\notag\\
	&+\int_0^te^{U_s}\Bigg[\eta^{(1)}\frac{(1-R^{(z,r)}_s)}{z}-\eta^{(2)}\frac{R^{(z,r)}_s}{z}\Bigg]ds+\int_0^{t}\int_{(0,\infty)}e^{U_s}\tilde{g}^{(z)}(R^{(z,r)}_s,w)\nu^{(1)}(dw)ds\notag\\&+\int_0^t\int_{(0,\infty)}e^{U_s}\tilde{h}^{(z)}(R^{(z,r)}_s,w)\nu^{(2)}(dw)ds\notag\\&+\int_0^te^{U_s}R^{(z,r)}_s(1-R^{(z,r)}_s)ds\left(\int_{[1,\infty]}\frac{w^2}{w+z}m^{(2)}(dw)-\int_{[1,\infty]}\frac{w^2}{w+z}m^{(1)}(dw)\right),\qquad t\geq0.
\end{align}
We define
for $t\geq 0$
\begin{align*}
	&A^{(6,z_n)}_t:=(\psi^{(2)}(0+)-\psi^{(1)}(0+))\frac{1}{\sqrt{z_n}}\int_0^te^{U_s}\left[\sqrt{z_n}(R^{(z_n,r)}_s-R^{(\infty,r)}_s)\right]^2ds.
\end{align*}
Using Skorohod's representation theorem (see \cite[Theorem 6.7]{Bill}) and bounded convergence, we obtain 
that 
\begin{align}\label{conv_fluc_1}
	A^{(6,z_n)}_t\to 0, \qquad \text{weakly as $n\to\infty$}, 
\end{align}
and hence the convergence also holds in probability. 
Now using \eqref{def_A_5}, straightforward computations give
\begin{align*}
	\E\left[\sup_{t\in[0,T]}|\sqrt{z_n}A^{(5,z_n)}_t|\right]&\leq \frac{1}{\sqrt{z_n}}\int_0^Te^{U_s}ds\sum_{i=1}^2\Bigg(2c^{(i)}+\int_{(0,\infty)}w^2m^{(i)}(dw)+\eta^{(i)}+\int_{(0,\infty)}w\nu^{(i)}(dw)\Bigg).
\end{align*}
Hence, 
\begin{align}\label{conv_fluc_2}
	\lim_{n\to\infty}\E\left[\sup_{t\in[0,T]}|\sqrt{z_n}A^{(5,z_n)}_t|\right]=0.
\end{align}
Now by \eqref{bound_conv_z_0} we have for $t\geq0$
\begin{align*}
	X^{(z_n)}_t=e^{U_t}\sqrt{z_n}(R^{(z_n,r)}_t-R^{(\infty,r)}_t)-\sqrt{z_n}A^{(5,z_n)}_t+A^{(6,z_n)}_t\blue{-}(\sqrt{z_n}A^{(4,z_n)}_t-X^{(z_n)}_t).
\end{align*}
Therefore, using \eqref{conv_fluc_1}, \eqref{conv_fluc_2}, and Lemma \ref{conv_A_2} we obtain that 
\begin{equation}
	X^{(z_n)}\to e^{U}Y^{(\infty)},\qquad\text{as $n\to\infty$ weakly in $\mathbb{D}([0,T],\R)$.}
\end{equation}
Hence, by an application of Lemma \ref{iden-limit} we obtain that $e^{U}Y^{(\infty)}\overset{\mathcal{L}}{=}X^{(\infty)}$,
which implies that the limit of any convergent subsequence of the family $\{\sqrt{z}(R^{(z,r)}-R^{(\infty,r)}):z\geq 1\}$ is equal in law to $e^{-U}X^{(\infty)}$. Therefore, 
\begin{align*}
	\sqrt{z}(R^{(z,r)}-R^{(\infty,r)})\to e^{-U}X^{(\infty)},\qquad \text{as $z\to\infty$,}
\end{align*}
weakly in $\mathbb{D}([0,T],\R)$.
\section{Moment duality for the $\Lambda$-asymmetric frequency process}\label{duality}
This section will study the relationship between the $\Lambda$-asymmetric frequency process  $R^{(z,r)}$, which was introduced in Section \ref{proc_r}, and a particular class of branching-coalescent processes.

For $t\geq0$, $x\in[0,1]$ and $n\in \mathbb{N}$, the expression $\E_x[X_t^n]$ can be understood as the probability of sampling $n$ individuals of type one in a random sample of $n$ individuals at time $t$, given that the frequency of type one individuals at time zero is $x$. If we imagine that each individual copies the type of its parent, all the ancestors of the sampled individuals have to be of type one in order for the whole sample to be of this type. Thinking in these terms, the equation $\E_x[X_t^n]=\E_n[x^{A_t}]$ allows us to interpret $A_t$ as the number of ancestors at time zero of the $n$ individuals sampled at time $t$, because $\E_n[x^{A_t}]$ would then be the probability that each of this ancestors is of type one. 

Because of this heuristics, in the introduction we speak of $A_t$ as the genealogy of $X_t$, but it is important to say that this is an abuse of terminology. The precise way to describe the relation between the forward and backward processes is as moment duals. This relation does not hold almost surely in general. In the case of the $\Lambda$-coalescent, the associated frequency process has a moment duality with the process that counts the number of ancestors. This is a pathwise duality and has an interpretation in terms of sampling \cite{D-K1,D-K2}. This can be generalized to include mutations and selection without losing the pathwise duality and the sampling interpretation, but the dual won't be precisely the process of the number of ancestors anymore. In some other cases, such as efficiency, understood as the difference in the resources consumed in order to reproduce, there is a moment dual, but in the literature there is no pathwise construction nor an explicit sampling duality.

This class consists of continuous time Markov chains taking values in $\mathbb{N}_0\cup\{\Delta\}$ (where $\mathbb{N}_0=\mathbb{N}\cup \{0\}$). The point $\Delta$ is a cementery state and we assume that $x^{\Delta}=0$ for all $x\in[0,1]$. 

For each $i,k\in\mathbb{N}_0$ with $i\geq k$, and $v\in[0,1]$ we define the following terms
\begin{align*}
	\lambda_{i,k}^l(v)&=\int_{(0,v)}\left[(1-u)^{i-k}u^k\right]\mathbf{T^{(z)}}(m^{(l)})(du),\\
	\mu_{i,k}^l&=\int_{(0,1)}\left[(1-u)^{i-k}u^k\right]\mathbf{T^{(z)}}(\nu^{(l)})(du),\qquad\text{$l=1,2.$}
\end{align*}
Now for each $i,j\in\mathbb{N}_0\cup\{\Delta\}$ let us consider the following set of real numbers
\begin{equation}\label{gendual}
	q^z_{ij}=
	\begin{cases} 
		\displaystyle{\overline{\mu}_{i,i}^1}&\mbox{if } \text{$i\in \N$ and $j=0$,} \\
		\displaystyle{si+\sum_{k=2}^i\kappa_k\binom{i}{k}}+\beta_i&\mbox{if } \text{$i\in \N$ and $j=i+1$,} \\
		\displaystyle \binom{i}{i-j}\overline{\mu}^1_{i,i-j}+\binom{i}{i-j+1}\overline{\lambda}^1_{i,i-j+1}  &\mbox{if } \text{$i\geq 2$ and $j\in\{1,..,i-1\}$,} \\
		\displaystyle \alpha_i & \mbox{if }  \text{$i\in \N$ and $j=\Delta$,} \\
		0&\mbox{otherwise,} 
	\end{cases}
\end{equation}
where 
\begin{itemize}
	\item For $2\leq k\leq i$,
	\begin{align*}
		\overline{\lambda}_{i,k}^{(1)}&=\int_{[0,1)}\left[(1-u)^{i-k}u^k\right]u^{-2}\Lambda^{(1)}(du),
	\end{align*}
	with $\Lambda^{(1)}(du)=\frac{2c^{(1)}}{z}\delta_0(du)+zu^2\mathbf{T^{(z)}}(m^{(1)})(du)$. These are the transition rates of the $\Lambda$-coalescent as first described in \cite{D-K1,Pitman, Sagitov}.
	\item For $1\leq k\leq i$
	\begin{align*}
		\overline{\mu}_{i,k}^{(1)}&=\int_{[0,1)}\left[(1-u)^{i-k}u^k\right]u^{-1}\Gamma^{(1)}(du),
	\end{align*}
	with $\Gamma^{(1)}(du)=\frac{\eta^{(1)}}{z}\delta_0(du)+u\mathbf{T^{(z)}}(\nu^{(1)})(du)$. These transitions have been recently found in the context of coordination, in particular they correspond to coordinated death (see \cite{GKT}).
	\item 
	\begin{align*}
		s&=\frac{2(c^{(1)}-c^{(2)})}{z}+(b^{(1)}-b^{(2)})\\
		&+z\left(\int_{(0,1/(1+z))}\frac{u^2}{1-u}\mathbf{T^{(z)}}(m^{(1)})(du)-\int_{(0,1/(1+z))}\frac{u^2}{1-u}\mathbf{T^{(z)}}(m^{(2)})(du)\right).
	\end{align*}
	The difference $b^{(1)}-b^{(2)}$ corresponds to classic Malthusian selection. The difference $c^{(1)}-c^{(2)}$ is Gillespie's selection of the variance. The last difference corresponds to a new form of selection that comes from big reproduction events.
	\item For $k\geq 2$
	\begin{align*}
		\kappa_k&=z\Bigg[k\left(\lambda_{i,k}^{(1)}\left(\frac{1}{1+z}\right)-\lambda_{i,k}^{(2)}\left(\frac{1}{1+z}\right)\right)-\left(\lambda_{i,k}^{(1)}\left(1\right)-\lambda_{i,k}^{(2)}\left(1\right)\right)\Bigg]+\frac{2(c^{(1)}-c^{(2)})}{z}1_{\{k=2\}}.
	\end{align*}
	The difference $c^{(1)}-c^{(2)}$ corresponds to pairwise branching and has been studied in the context of efficiency \cite{GMP, GPP}. All the other terms are new, and their interpretation will be clarified in future work.
	\item For $k\geq 1$
	\begin{align*}
		\beta_k=-kz\left[\left(\lambda_{k,1}^{(1)}\left(1\right)-\lambda_{k,1}^2\left(1\right)\right)-\left(\lambda_{k,1}^{(1)}\left(\frac{1}{1+z}\right)-\lambda_{k,1}^2\left(\frac{1}{1+z}\right)\right)\right].
	\end{align*}
	\item For $k\geq 1$
	\begin{align*}
		\alpha_k=\int_{[0,1)}(1-(1-u)^k)u^{-1}\rho^{(2)}(du),
	\end{align*}
	with $\rho^{(2)}(du)=\frac{\eta^{(2)}}{z}\delta_0(du)+u\mathbf{T^{(z)}}(\nu^{(2)})(du)$.
\end{itemize}
In the case that $q_{ij}^z\geq0$ for every $i,j\in\mathbb{N}_0\cup\{\Delta\}$, we will define a $\N_0\cup \{\Delta\}$-valued continuous Markov chain $N^{(z,n)}=\{N^{(z,n)}_t:t\geq0\}$ starting from $n\in\mathbb{N}$, whose generator is given by $\mathcal{Q}^{(z)}=(q^{z}_{ij})_{i,j\in\mathbb{N}}$. 
It is important to note that the assumption that all the rates are positive restricts the cases one can study via duality to situations in which one type has selective advantage over the other. This is the case of the ancestral selection graph, or neutral cases in which one type has a smaller expected fixation time, like the Wright Fisher model with efficiency.

The rest of this section is devoted to showing that the moment dual of the frequency process $R^{(z,r)}$ is the continuous-time Markov chain $(N^{(z,n)}_t)_{t\geq0}$. An effective procedure to prove the moment duality is to use their infinitesimal generators. The following proposition is a direct consequence of Theorem 4.11 in Chapter 4 of Ethier Kurtz \cite{EK} taking $H$ bounded and continuous, and $\alpha=\beta=0$, and can also be seen as a
small modification of Proposition 1.2 of Jansen and Kurt in \cite{JK}.
\begin{proposition}\label{prop:provingduality}
	Let $Y^{(1)}=\{Y^{(1)}_t: t\geq 0\}$ and $Y^{(2)}=\{Y^{(2)}_t: t\geq 0\}$ be two Markov processes taking values on $E_1$ and $ E_2$, respectively. Let $H:E_1\times E_2\rightarrow \mathbb{R}$ be a bounded and continuous function and assume that there exist functions $g_i:E_1\times E_2\mapsto \mathbb{R}$, for $i=1,2$, such that for every $n\in E_1,x\in E_2$ and every $T>0$,  the processes $M^{(1)}=\{M_t^{(1)}: 0\le t\le T\}$ and $M^{(2)}=\{M_t^{(2)}, 0\le t\le T\}$, defined as follows
	\begin{eqnarray}
		M_t^{(1)}:=H(Y^{(1)}_t,x)&-&\int_0^t g_1(Y^{(1)}_s,x)\ud s\label{MG1}\\
		M_t^{(2)}:=H(n,Y^{(2)}_t)&-&\int_0^t g_2(n, Y^{(2)}_s)\ud s \label{MG2}
	\end{eqnarray}
	are martingales with respect to  the natural filtration of $Y^{(1)}_t$ and $Y^{(2)}_t$, respectively. Then, if 
	$$
	g_1(n,x)=g_2(n,x)\qquad \textrm{for all}\qquad  n\in E_1,x\in E_2,
	$$
	the processes $Y^{(1)}$ and $Y^{(2)}$ are dual with respect to $H$. 
\end{proposition}
The previous result provides a general duality relationship between two Markov process. In our case we are  interested in the particular case of moment duality which follows from Proposition \ref{prop:provingduality} by taking $H(n,x):=x^n$, for $x\in [0,1]$ and $n\in \mathbb{N}_0$. The proof of the next result is deferred to Appendix \ref{App_6}.
\begin{theorem}\label{theo_dual}
	Assume that $q_{ij}\geq 0$ for every $i,j\in \N_0\cup \{\Delta\}.$ Then, for every $r\in [0,1]$, $n\in \N_0\cup \{\Delta\}$ and $t>0$
	\begin{align*}
		\E[(R_t^{(z,r)})^n]=\E[r^{N_t^{(z,n)}}].
	\end{align*}
\end{theorem}

\section{The space of CSBPs is homeomorphic to the spaces of $\Lambda$-coalescents}\label{CBduality}
In this section, we will assume that the processes $X^{(1)}$ and $X^{(2)}$ given in \eqref{CBI_SDE} correspond to equally distributed CB processes with characteristic triplet $(b^{(1)},c^{(1)},m^{(1)})$ (i.e. $\xi^{(i)}=0$ for $i=1,2$). After an application of the culling procedure at the level $z>0$ to the two-dimensional process $(R,Z)$ given by \eqref{tpsp} and \eqref{fp}, we obtain, by Proposition \ref{feller}, that the frequency process $R^{(z,r)}$ given in \eqref{sde_p_cropped} has an infinitesimal generator given for any $f\in\mathcal{C}^2([0,1])$ by
\begin{align*}
	\mathcal{L}^{(z)}f(r)&=c^{(1)}\frac{r(1-r)}{z}f''(r)\notag\\&+z\int_{(0,1)}\Bigg[rf\left(r(1-u)+u\right)+(1-r)f\left(r(1-u)\right)-f\left(r\right)\Bigg]\mathbf{T^{(z)}}(m^{(1)})(du),
\end{align*}
and therefore the process $R^{(z,r)}$ corresponds to the classic $\Lambda$-frequency process, whose dual is the block counting process of a $\Lambda$-coalescent. Indeed, by Theorem \ref{theo_dual} we have that the associated moment dual $N^{(z,r)}$ has a generator $Q^{z}=(q^{z}_{ij})_{i,j\in\mathbb{N}}$ given by
\begin{equation*}
	q_{ij}=
	\begin{cases} 
		\displaystyle \binom{i}{i-j+1}\overline{\lambda}^1_{i,i-j+1}  &\mbox{if } \text{$i\geq 2$ and $j\in\{1,..,i-1\}$,} \\
		0&\mbox{otherwise,} 
	\end{cases}
\end{equation*}
where for $2\leq k\leq i$,
\begin{align*}
	\overline{\lambda}_{i,k}^{(1)}&=\int_{[0,1)}\left[(1-u)^{i-k}u^k\right]u^{-2}\Lambda^{(1)}(du),
\end{align*}
with $\Lambda^{(1)}(du)=\frac{2c^{(1)}}{z}\delta_0(du)+zu^2\mathbf{T^{(z)}}(m^{(1)})(du)$.

Hence, using this procedure, it is natural to map any CB process with the characteristic triplet $(b^{(1)},c^{(1)},m^{(1)})$ to the $\Lambda$-coalescent with associated measure given by $\Lambda^{(1)}$, which can be understood as the genealogy of the CB process $X^{(1)}$. We observe that under this mapping, all of the CB processes with the same diffusion term and jump measure are mapped to the same $\Lambda$-coalescent. Therefore, we will consider the previous mapping from the quotient space obtained by using the equivalence relation in which two CB processes are
related if and only if they have the same diffusion term and the same L\'evy measure to the space of $\Lambda$-coalescents. In this section, we will show that this mapping from the quotient space of CB processes to the genealogy associated with each class is a homeomorphism.

Our strategy in this section is to first show that if the sequence of characteristic triplets associated with a sequence of CB processes converges to the characteristic triplet of some CB process suitably, then the sequence of CB processes converge. We then show that if a sequence of finite measures on $[0,1]$ converges to another such measure, then the sequence of their associated $\Lambda$-coalescents also converges. These two results induce an easy to check equivalent reformulation of our desired result: \textit{the map that sends  characteristic triplets of CB processes to finite measures characterizing $\Lambda$-coalescents, induced by sending each CB process to its genealogy, is a homeomorphism}. This is proven in the final step of the proof.

We will denote by $\mathcal{LM}(\R_+)$ the space of L\'evy measures on $(0,\infty)$; that is, a positive measure $m$ belongs to $\mathcal{LM}(\R_+)$ if and only if it satisfies the condition $\int_{(0,\infty)}(1\wedge x^2)m(dx)<\infty$.

Now, let us consider $\Psi$ the space of CB  processes, we have seen by (\ref{bran_mech}), that each element $Z\in\Psi$ can be characterized in terms of its branching mechanism $\psi$, and therefore by its associated  triplet $(b,c,m)\in\R\times\R_+\times\mathcal{LM}(\R_+)$. 

We now provide a criterion for the convergence of a sequence of CB processes in terms of the convergence of the associated sequence of characteristic triplets. Hence, following pg. 244 in \cite{Ka}, for each triplet $(b,c,m)\in\R\times\R_+\times\mathcal{LM}(\R_+)$ we define
\begin{align*}
	\tilde{b}:&=b+\int_{\R\backslash\{0\}}\left(\frac{x}{x^2+1}-x1_{\{|x|\leq 1\}}\right)m(dx),\\
	\tilde{m}:&=c\delta_0(dx)+\frac{x^2}{x^2+1}m(dx).
\end{align*}
In the space of triplets $\R\times\R_+\times\mathcal{LM}(\R_+)$ we introduce the following metric:
\begin{align*}
	d_{\Psi}((b^{(1)},c^{(1)},m^{(1)}),(b^{(2)},c^{(2)},m^{(2)})):&=|\tilde{b}^{(1)}-\tilde{b}^{(2)}|+\rho(\tilde{m}^{(1)},\tilde{m}^{(2)}),
\end{align*}
where $\rho$ denotes the Prohorov distance in the space of finite measures.

We recall the Skorohod topology on the space of cadlag functions from $\R_+$ to $\R_+$: a sequence $(f_n)_{n\geq1}$ converges to $f$ in the Skorohod topology if there exists a sequence of homeomorphisms $(\lambda_n)_{n\geq 1}$ of $\R_+$ into itself such that
\[
f_n-f\circ\lambda_n\qquad \lambda_n\to \text{Id},\qquad \text{uniformly on compact sets.}
\]
Additionally, we consider the uniform Skorohod topology introduced in \cite{CLU}. Consider a distance $d$ on $[0,\infty]$, which makes it homeomorphic to $[0,1]$. We say that a sequence $(f_n)_{n\geq1}$ converges to $f$ in the uniform Skorohod topology if there exists a sequence of homeomorphisms $(\lambda_n)_{n\geq 1}$ of $\R_+$ into itself such that
\[
d(f_n,f\circ\lambda_n)\to 0\qquad \lambda_n\to \text{Id},\qquad \text{uniformly on $\R_+$.}
\]
We first provide some auxiliary results that will be needed in the proof of our main result.
\begin{proposition}\label{cont2}
	Let $\{Z^n\}$ be a sequence of continuous-state branching processes with the characteristic triplets $(b_n,c_n,m_n)$. Additionally, consider a continuous-state branching process $Z$ with the characteristic triplet $(b,c,m)$. Assume that 
	\[
	\lim_{n\to\infty}d_{\Psi}((b_n,c_n,m_n),(b,c,m))=0.
	\]
	Then $Z^n\rightarrow Z$ as $n\to\infty$, weakly on the space of cadlag paths from $\R_+$ to $[0,\infty]$ with the Skorohod topology if the branching mechanism $\psi$ of $Z$ is nonexplosive, and with the uniform Skorohod topology if $\psi$ is explosive.
\end{proposition}
\begin{proof}
	For each $n\in\mathbb{N}$ consider a spectrally positive L\'evy process $X^n$ with the characteristic triplet $(b_n,c_n,m_n)$; that is, the Laplace exponent of $X^n$ is given by
	\begin{align*}
		\log \E\left[e^{-\lambda X^n_t}\right]=b_n\lambda+c_n\lambda^2 +\int_{(0,\infty)}(e^{-\lambda z}-1+z\lambda x1_{(0,1)}(z))m_n(dz),\qquad\lambda\geq 0.
	\end{align*}
	Then, by using Lemma 13.15 in \cite{Ka} together with the fact that
	\[\lim_{n\to\infty}d_{\Psi}((b_n,c_n,m_n),(b,c,m))=0,\] 
	we obtain that
	\begin{equation}\label{Levycov}
		X^n\Rightarrow X,\qquad\text{ as $n\to\infty$,}
	\end{equation}
	weakly in the space of cadlag paths form $\R_+$ to $\R_+$ endowed with the Skorohod topology, where $X$ is a L\'evy process with the characteristic triplet $(b,c,m)$.
	
	By the continuity of the Lamperti transform, as in Corollary 6 in \cite{CPU}, together with \eqref{Levycov} we obtain that
	\[
	Z^n\Rightarrow Z,\qquad\text{ as $n\to\infty$,}
	\]
	weakly on the space of cadlag paths from $\R_+$ to $[0,\infty]$ endowed with the Skorohod topology if the branching mechanism $\psi$ of $Z$ is nonexplosive, and with the uniform Skorohod topology if $\psi$ is explosive.
\end{proof}


For our next result, we denote the space of finite measures on $[0,1]$ by $\mathcal{M}_F([0,1])$ and let $(\mathcal{P},d)$ be the space of partitions of the natural numbers endowed with the distance $d$, which is defined for any two partitions $\pi,\pi'\in\mathcal{P}$ by
$$
d(\pi,\pi')= M^{-1}\text{ if and only if } \pi|_{[M]}=\pi'|_{[M]}\text{ and }\pi|_{[M+1]}\neq\pi'|_{[M+1]}
$$
where $\pi|_{[M]}$ is the restriction of $\pi$ to $[M]=\{1,2,...,M\}$. With this, we mean that given $\pi$, then $\pi|_{[M]}$ is the partition of $[M]=\{1,2,...,M\}$ constructed by the rule $i,j\in [M]$ are in the same block on $\pi|_{[M]}$; that is, $i\sim j$ in $\pi|_{[M]}$, if $i\sim j$ in $\pi$. Similarly, $\mathcal{P}|_{[M]}$, for any $M\in\N$, to be the set of partitions of $[M]=\{1,2,..., M\}$ equipped with the same distance $d.$ Note that any element $\pi\in\mathcal{P}$ induces a partition of $[M]$ and denote such partition by $\pi|_{[M]}$.

Consider $A\subset \mathcal{P}$ and $\varepsilon>0$. Denote, by $A^{\varepsilon}$, to the $\varepsilon$-neighborhood of $A$, i.e.
\[
A^{\varepsilon}:=\{\pi\in \mathcal{P}: d(\pi,\pi')<\varepsilon, \text{ for some }\pi'\in A\}.
\]

Observe that its neighbourhoods in $(\mathcal{P},d)$ are characterized by the rule $\pi\in A^{1/M}$ if and only if there exists an element of $A$ whose restriction to $[M]$ agrees with the restriction of $\pi$ to $[M]$. This is,
\begin{align}
	A^{1/M}
	=\{\pi\in \mathcal{P}: \pi|_{[M]}=\pi'|_{[M]},\text{ for some }\pi'\in A\}.\label{Neig}
\end{align} 
Finally, we denote $\mathcal{D}_M=\{\pi\in \mathcal{P}:\{M+1,M+2,...\}\in \pi\}$, the partitions that have a block consisting of all the integers larger than $M$, and note that $\mathcal{D}=\cup_{i=1}^\infty \mathcal{D}_i$ is a countable dense set. This implies that $(\mathcal{P},d)$ it is separable, which in turn implies that the Prohorov metric can be used to study weak convergence of stochastic processes with trajectories in $(\mathcal{P},d)$.
\begin{proposition}\label{cont1}
	Let $\{\Pi^N\}_{N\in\N_0}$ be a sequence of $\Lambda$-coalescents  with characteristic measures $\{\Lambda^N\}_{N\in\N_0}\subset\mathcal{M}_F([0,1])$, such that $\Lambda^N\rightarrow \Lambda^0$ weakly as $N\to\infty$. Let $\Pi^0$ be the $\Lambda$-coalescent associated to $\Lambda^0$. Then 
	$$
	\Pi^N\rightarrow \Pi^0,\qquad \text{as $N\to\infty$,}
	$$ 
	weakly in $\mathbb{D}(\R_+, (\mathcal{P},d))$.
\end{proposition}
\begin{proof}
	As before, we denote  $\Pi|_{[M]}$ the sequence of partitions  induced by $\Pi$.  It is known that if $\Pi$ has characteristic measure $\Lambda$, then $\Pi|_{[M]}$ is a Markov chain which at state $\pi$ can jump to the state $\pi'$ if there exist $i\leq k:=|\pi|$ such that $\pi'$ can be constructed by merging $i$ blocks of $\pi$. In this case, it jumps from $\pi$ to $\pi'$ at rate $\tilde{\lambda}_{k,k-i+1}$ where
	$$
	\tilde{\lambda}_{k,i}=\int_0^1u^i(1-u)^{k-i}\frac{\Lambda(du)}{u^2}, \qquad\text{for $2\leq i\leq k$.}
	$$
	Because $u^{i-2}(1-u)^{k-i}$ is a bounded and continuous function for every $1<i\leq k$, and $k>1$, the fact that $\Lambda^N\rightarrow \Lambda^0$ weakly as $N\to\infty$, implies that the transitions of the processes $\{\{\Pi^N_t|_{[M]}, t>0\}\}_{N\in\N}$ converge to the transitions of the process $\{\Pi^0_t|_{[M]}, t>0\}$, for every $M\in\N$.
	Using that $M\in\N$, $\{\Pi_t|_{[M]}, t>0\}$ is a continuous-time Markov chain with a finite state-space, we have that the convergence of their transitions implies that  
	$$\{\Pi^N_t|_{[M]}, t>0\}\rightarrow \{\Pi^0_t|_{[M]}, t>0\}$$ 
	weakly as $N\to\infty$ in the space of cadlag paths from $\R_+$ to $(\mathcal{P}|_{[M]},d)$ with the Skorohod topology. Using that the state space $\mathcal{P}|_{[M]}$ is finite, the convergence of the restricted processes and the Skorohod representation theorem, we see that in some probability space
	$
	\lim_{N\rightarrow \infty}\p(\Pi^N_t|_{[M]}=\Pi^0_t|_{[M]}, \forall t\in[0,T])=1
	$
	for all $T>0.$  Because $\{1,2,...,M\}$ is the only absorbing state and it is reached in finite time, we can strengthen this to 	
	$
	\lim_{N\rightarrow \infty}\p(\Pi^N_t|_{[M]}=\Pi^0_t|_{[M]}, \forall t>0)=1.
	$
	
	For any partition $\pi$ we denote its restriction to $[M]$ by  $\pi|_{[M]}$, defined by the rule that that for any $i,j\in [M]$, $i\sim j$ in $\pi|_{[M]}$ if $i\sim j$ in $\pi$. If a $\mathcal{A}\subset \mathcal{P}$, we define its restriction by  $\mathcal{A}|_{[M]}:=\{\pi\in \mathcal{P}|_{[M]}:\pi=\pi'|_{[M]} \text{for some $\pi'\in \mathcal{A}$} \}$.
	
	Take $N$ such that  $\p(\Pi^N_t|_{[M]}=\Pi^0_t|_{[M]})>1-1/M$. Then, for any measurable set $\mathcal{A}\subset \mathbb{D}(\R_+, (\mathcal{P},d))$  	
	\begin{equation}\label{cool}
		\p(\Pi^N\in \mathcal{A})\leq \p(\Pi^N|_{[M]}\in \mathcal{A}|_{[M]})\leq \p(\Pi|_{[M]}^0\in \mathcal{A}^{1/M}|_{[M]})+1/M=\p(\Pi^0\in \mathcal{A}^{1/M})+1/M.
	\end{equation}
	Where in the first inequality we used the containment of events $\{\Pi^N\in \mathcal{A}\}\subset\{\Pi^N|_{[M]}\in \mathcal{A}|_{[M]}\}$, in the second we used the definition of the Prohorov's distance and in the equality the characterization of the neighbourhoods in  $(\mathcal{P}, d)$ (as discussed just before the statement of this result).
	
	From Equation \eqref{cool}, we conclude that $\rho( \p(\Pi^N \in \cdot), \p(\Pi \in \cdot))<1/M,$ where $\rho$ is the Prohorov metric and $M$ is arbitrary. Thus, the proof is complete. 
\end{proof}
\begin{theorem}\label{homeomorphism}
	Consider the metric space \textbf{L} of the laws of $\Lambda$-coalescents with no atom at $\{1\}$ equipped with the Prohorov distance over the space of probability measures defined on the space $\mathbb{D}(\R_+, (\mathcal{P},d))$. In addition, for $r\in\R$, consider the space  $\mathbf{\Psi}_r\subset\mathbf{\Psi}$ of CB processes with $\tilde b=r$ equipped with the Prohorov distance over the space of probability measures defined on the space $\mathbb{D}([0,T],\R_+)$ endowed with the uniform Skorohod topology. 
	Then, \textbf{L} and $\mathbf{\Psi}_r$  are homeomorphic.

	Furthermore, consider the mapping $\mathbf{H^{(z)}}:\mathbf{\Psi}_r\mapsto \mathbf{L}$ such that a CB  process with the triplet $(b,c,\nu)$ is mapped to the $\Lambda$-coalescent with the associate measure 
	$$
	\mathbf{H^{(z)}}((b,c,\nu))=\frac{c}{z}\delta_0+zy^2\mathbf{T^{(z)}}(\nu).
	$$
	Then, for every $z>0,$  $\mathbf{H^{(z)}}$ is a homeomorphism, with inverse $\mathbf{H^{(z)}}^{-1}$ sending a $\Lambda$-coalescent to the CB process with characteristic triplet
	\begin{equation}\label{image}
		\left(r-\int_{\R\backslash\{0\}}\left(\frac{x}{x^2+1}-x1_{\{|x|\leq 1\}}\right)\mathbf{(T^{(z)})^{-1}(\Lambda)(dx)},z\Lambda(\{0\}), (zy^2)^{-1} \mathbf{(T^{(z)})^{-1}}(\Lambda-\Lambda(\{0\})\delta_0)\right)
	\end{equation}
\end{theorem}
\begin{proof}
	First we show that the mapping $\mathbf{H^{(z)}}$ is one-to-one and onto. 
	
	\textit{Onto)} Chose an arbitrary finite measure $\Lambda$ and note that the branching process with triplet specified in \eqref{image} is mapped under $\mathbf{H^{(z)}}$ to the coalescent with characteristic measure $\Lambda$.
	
	\textit{One-to-one)} Assume that $\mathbf{H^{(z)}}((b_1,c_1,\nu_1))=\mathbf{H^{(z)}}((b_2,c_2,\nu_2))$, then
	\[
	\frac{c_1}{z}=\mathbf{H^{(z)}}((b_1,c_1,\nu_1))(\{0\})=\mathbf{H^{(z)}}((b_2,c_2,\nu_2))(\{0\})=\frac{c_2}{z},
	\]
	which implies that $c_1=c_2$. Now consider the measurable function $f:(0,\infty)\to\R$ and write $w=\frac{yz}{1-y}$. We observe that
	\begin{align*}
		\int_{(0,\infty)}f(w)\nu_1(dw)&=\int_{(0,1)}f\left(\frac{yz}{1-y}\right)(zy^{2})^{-1}zy^2\mathbf{T^{(z)}}(\nu_1)(dy)\\&=\int_{(0,1)}f\left(\frac{yz}{1-y}\right)(zy^{2})^{-1}zy^2\mathbf{T^{(z)}}(\nu_2)(dy)=\int_{(0,\infty)}f(w)\nu_2(dw).
	\end{align*}
	So we conclude that $\nu_1=\nu_2$.
	
	To proceed with the proof, we introduce the following notation
	\begin{align*}
		(b_{\Lambda},c_{\Lambda},\nu_{\Lambda}):&=	\Bigg(r-\int_{\R\backslash\{0\}}\left(\frac{x}{x^2+1}-x1_{\{|x|\leq 1\}}\right)\mathbf{(T^{(z)})^{-1}}(\Lambda)(dx),\notag\\&\hspace{5cm}z\Lambda(\{0\}), (zy^2)^{-1} \mathbf{(T^{(z)})^{-1}}(\Lambda-\Lambda(\{0\})\delta_0)\Bigg).
	\end{align*}
	Then, by noticing that
	\[
	\tilde{b}_{\Lambda}:=b_{\lambda}+\int_{\R\backslash\{0\}}\left(\frac{x}{x^2+1}-x1_{\{|x|\leq 1\}}\right)\mathbf{(T^{(z)})^{-1}}(\Lambda)(dx)=r,
	\]
	we obtain that the CB process with characteristic triplet $(b_{\Lambda},c_{\Lambda},\nu_{\Lambda})$ belongs to $\mathbf{\Psi}_r$. On the other hand, by the definition of the mapping $\mathbf{H^{(z)}}$ we obtain that
	\[
	\mathbf{H^{(z)}}\left((b_{\Lambda},c_{\Lambda},\nu_{\Lambda})\right):=\Lambda.
	\]
	Hence, $\mathbf{H^{(z)}}^{-1}(\Lambda)=(b_{\Lambda},c_{\Lambda},\nu_{\Lambda})$.
	
	Finally, because $\mathbf{H^{(z)}}$ and its inverse are continuous, which is the content of Propositions \ref{cont2} and \ref{cont1}, the proof is complete.
\end{proof}

\section{The asymmetric Eldon-Wakely coalescent: A minimalistic example}\label{Ex}

To illustrate our results, we study a simple example heuristically. Fix parameters $z,v_1,v_2>0,$ and $y_1,y_2\in (0,1),$ the \textit{simple $\Lambda$-asymmetric frequency process} is the solution to the SDE
\begin{align}\label{SGlambda}
	dR^{(z,r)}_t=
	&\int_{(0,\infty)} y_1(1-R^{(z,r)}_{t-})1_{\{v<zR^{(z,r)}_{t-}\}} N_1(dt,dv)\notag\\&\hspace{3cm}-\int_{(0,\infty)} y_2R^{(z,r)}_{t-}1_{\{v<z(1-R^{(z,r)}_{t-})\}} N_2(dt,dv),\qquad t>0,\notag\\
	R^{(z,r)}_0=&r\in[0,1].
\end{align}
where for $i=1,2$, $N_i(dt,dv)$ are independent Poisson random measures on the space $[0,\infty)\times(0,1)$ with intensity measures $zv_ids\times dv$ for $i=1,2. $

Let $X^{(1)}$ and  $X^{(2)}$ be two CB processes, such that the only transitions of  $X^{(i)}$  are jumps of size $w_i>0$ that occur at rate $x v_i$ when the process is at the state $x$ for $i=1,2$. More formally, for each $i=1,2$, let $N^{(i)}=\{N_t^{(i)}:t\geq 0\}$ be a Poisson process with intensity parameters $v_i>0$, and define $Y_t^{(i)}=w_i N^{(i)}_t$ for $t\geq0$. Then, we define the CB process $X^{(i)}$ by means of the Lamperti transform; that is,
\[
X_t^{(i)}=Y^{(i)}_{\int_0^t X^{(i)}_s ds}, \qquad \text{$t\geq0$, $i=,1,2$}. 
\]
\begin{figure}[h]
	\begin{center}
		\includegraphics[height=.3\textwidth]{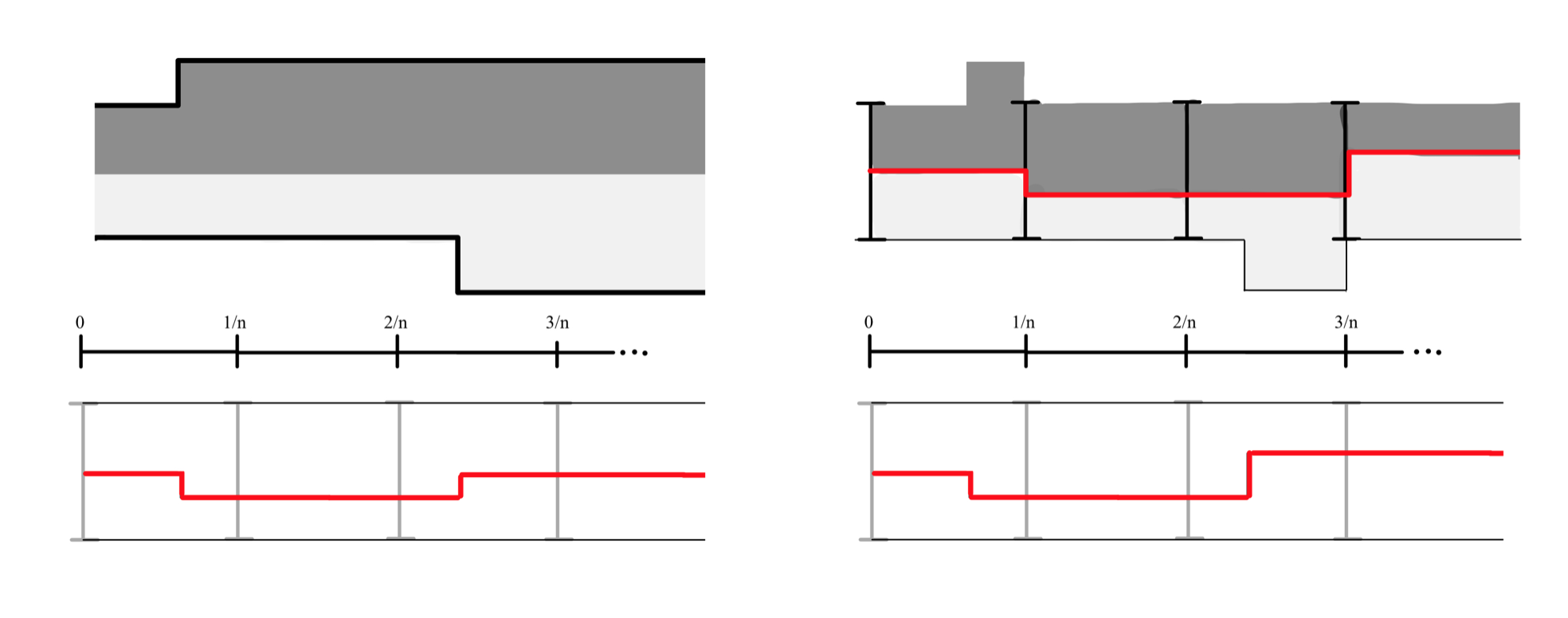}
		\caption{\footnotesize{ A realization of the $\Lambda$-asymmetric Eldon-Wakely frequency process, starting from two simple CB processes. In the upper left corner the process $X^{(1)}$ (dark grey) and the process $X^{(2)}$ (light gray) are depicted. $X^{(1)}$  performs jumps of size $w_1$ at rate $v_1$ (the first jump in the picture), while $X^{(2)}$  performs jumps of size $w_2$ at rate $v_2$ (the second jump). The total mass process $Z$ is the sum of the two CB processes. In the lower left corner we draw the frequency process. Note that at the first jump, the frequency process makes a jump of size $y_1=T_z(w_1)=w_1/(z+w_1)$, where $z=X^{(1)}_0+X^{(2)}_0$}. However, at the second jump the total mass is no longer $z$ and the jump of the frequency process is not $T_z(w_2)$. This is an indication that the frequency process is not a Markov process. In the right-side of the figure, we observe how the Gillespie's culling procedure allows us to overcome this difficulty. At each sampling point, the total mass is returned to $z$, while the frequency is unchanged. Sampling points occur so often that with probability tending to one no more than one jump occurs between subsequent sampling times. Thus, the jumps of the CB process are always pushed forward to jumps of the frequency process by means of the function $T_z$ and the frequency process is Markovian at the sampling times.}
	\end{center}
\end{figure}

If we take $z=x_0^{(1)}+x_0^{(2)}$ and $r=x_0^{(1)}/z$, then we will show that the associated $\Lambda$-asymmetric frequency process $R^{(z,r)}$ is the limit of the culling procedure at level $z$ introduced in Section \ref{NC}. To this end, let $y_i=w_i/(z+w_i)$ and note that, at the position $(z,r)$, $X^{(1)}$ jumps at rate $rz v_1$. At each jump of $X^{(1)}$ the associated frequency process $R$ as defined in \eqref{fp}, will jump to the level 
$$(x_1+w_1)/(z+w_1)=\frac{x_1}{z}(1-\frac{w_1}{z+w_1})+\frac{w_1}{z+w_1}=r(1-y_1)+y_1=r+y_1(1-r).$$
Meanwhile, at level $(z,r)$, $X^{(2)}$ jumps at rate $(1-r)zv_2$ and $R$ jumps to the state $r-y_2r$. From the previous computations and applying the culling procedure in Section \ref{NC}, it is not difficult to show that for $T>0$
$$
\overline{R}^{(n,z)}\rightarrow R^{(z,r)},\qquad \text{as $n\to\infty$ weakly in $\mathbb{D}([0,T],[0,1])$  .}
$$
where $R^{(z,r)}$ is the $\Lambda$-asymmetric frequency process in \eqref{SGlambda} with parameters $z,v_1,v_2,y_1,y_2>0$, and  $\overline{R}^{(n,z)}$ is the jump Markov process with generator \eqref{dis_gen} obtained by the culling procedure introduced in Section \ref{NC}. This is a particular example of Theorem \ref{theo_cull}.

We are now interested in finding the moment dual of $R^{(z,r)}$; that is, we will construct the process $N^{(z,n)}$ such that for every $t>0, r\in[0,1]$ and $n\in\N$
$$
\E[r^{N^{(z,n)}_t}]=\E[(R^{(z,r)}_t)^n].
$$

Let $A_n(v_1,v_2,y_1,y_2):= v_2(1-(1-y_2)^n)-v_1(1-(1-y_1)^n)$ and assume $A_n(v_1,v_2,y_1,y_2)>0$. We will call $A_n(v_1,v_2,y_1,y_2)$ the difference between the total activities, for reasons that will become clear later on. Our model will reveal that $A_n(v_1,v_2,y_1,y_2)$ is, in some sense, the term under evolutionary selection.

The block counting process of the simple $\Lambda$-asymmetric coalescent with parameters $z,v_1,v_2,y_1,y_2$, is the asymmetric version of the Eldon-Wakely-Coalescent \cite{EW}, which is the coalescent arising from reproduction events with constant size. This is the $\N$ valued process $N^{(z)}=\{N^{(z)}_t: t\geq0\}$ with generator
\begin{equation*}
	q_{ij}=
	\begin{cases} 
		\displaystyle z v_1 \binom{i}{i-j+1}(1-y_1)^{j-1}y_1^{i-j+1} &\mbox{if } \text{$i\geq 2$ and $j\in\{1,..,i-1\}$,} \\
		\displaystyle z A_j(v_1,v_2,y_1,y_2)&\mbox{if } \text{$i\in \N$ and $j=i+1$,} \\
		0&\mbox{otherwise.} 
	\end{cases}
\end{equation*}
It is not difficult to see that the previous transitions correspond to those given in \eqref{gendual} with $\alpha_i=0$ for $i\geq 1$, $\overline{\mu}^1_{i,k}=0$ for $1\leq k\leq i$, and 
\begin{itemize}
	\item For $2\leq k\leq i$, 
	\[
	\overline{\lambda}_{i,k}^{(1)}=zv_1(1-y_1)^{i-k}y_1^k.
	\]
	\item $s=v_2y_21_{(0,1/1+z)}(y_2)-v_1y_11_{(0,1/1+z)}(y_1)$.
	\item For $k\geq 2$
	\begin{align*}
		\kappa_k&=z\Bigg[v_1(1-y_1)^{i-k}y_1^k(k1_{(0,1/1+z)}(y_1)-1)-v_2(1-y_2)^{i-k}y_2^k(k1_{(0,1/1+z)}(y_2)-1)\Bigg].
	\end{align*}
	\item For $k\geq 1$
	\begin{align*}
		\beta_k=-kz\left[v_1(1-y_1)^{k-1}y_1(1-1_{(0,1/1+z)}(y_1))-v_2(1-y_2)^{k-1}y_2(1-1_{(0,1/1+z)}(y_2))\right].
	\end{align*}
\end{itemize}
Note that in the first line we have the transitions of a $\Lambda$-coalescent with $\Lambda=v_1\delta_{y_1}$. Interestingly, only in the second line do we see the parameters $v_2$ and $y_2$, which are causing branching events that account for the asymmetry between the upper and lower jumps. If $v_1=v_2$ and $y_1=y_2$, then the second line is zero and we are left with the Eldon Wakely coalescent with $\Lambda=v_1\delta_{y_1}$. 

It is also surprising that the branching coefficient is in terms of $A_n(v_1,v_2,y_1,y_2)$ and 
that $ v_1(1-(1-y_1)^n)$ is the rate at which one observes an event of any type in an Eldon Wakely coalescent with $\Lambda=v_1\delta_{y_1}$. The fact that branching is related to selection allows us to state, in the spirit of Gillespie, that reproduction mechanisms 
are more likely to go to fixation if they have a larger total activity. 

It is possible to use standard techniques to show that $N^{(z,n)}$ is the moment dual of $R^{(z,r)}$. The generator of $R^{(z,r)}$, applied to any $f\in\mathcal{C}^2([0,1])$ is given by
\begin{eqnarray*}
	\mathcal{L}^{(z)}f(r)=v_1zr[f(r+y_1(1-r))-f(r)]+v_2z(1-r)[f(r-y_2r)-f(r)].
\end{eqnarray*}
By choosing $f_n(x)=x^n$ as a test function, we observe that 
\begin{eqnarray}\label{dual_exam}
	\mathcal{L}^{(z)}f_n(r)&=&zv_1r[(r+y_1(1-r))^n-r^n]+zv_2(1-r)[(r-y_2r)^n-r^n]\nonumber\\
	&=&zv_1r(r+y_1(1-r))^n-v_1r^n+(1-r)r^n[v_1-v_2]+zv_2(1-r)(r-y_2r)^n\nonumber\\
	&=&zv_1\sum_{k=2}^n\binom{n}{k}(1-y_1)^{n-k}y_1^k[r^{n-k+1}-r^n]+v_1(1-y_1)^n[r^{n+1}-r^n]\nonumber\\ 
	&&+zv_2(1-r)r^n(1-y_2)^n+(v_2-v_1)[r^{n+1}-r^n]\nonumber\\
	&=&zv_1\sum_{k=2}^n\binom{n}{k}(1-y_1)^{n-k}y_1^k[r^{n-k+1}-r^n]\nonumber\\  
	&&+[zv_2(1-(1-y_2)^n)-zv_1(1-(1-y_1)^n)][r^{n+1}-r^n]=\mathcal{Q}^{(z)}f_r(n),
\end{eqnarray}
where $\mathcal{Q}^{(z)}$ is the generator of $N^{(z,n)}$ and $f_x(n)=x^n$. This is a special case of Theorem \ref{theo_dual}. 

If we take $v=v_1=v_2$ and $w=w_1=w_2=yz/(1-y)$ , we  note that this implies $y=w/(z+w)$. This confirms the fact that the culling procedure at level $z$ and the duality relationship maps the CB process with the characteristic triplet $(0,0, v\delta_w)$ (as an element of $\mathbf{\Psi}_{\tilde{w}}$ where $\tilde{w}=w/(w^2+1)-w1_{\{|w|\leq 1\}}$) to the $\Lambda$-coalescent with $\Lambda=v\delta_y$. This confirms the result in Theorem \ref{homeomorphism}, where we additionally showed that this is a homeomorphism of metric spaces.

However, this is not true if we do not restrict ourselves to equally distributed CB processes. Indeed, let us consider $s>0$ and the following CB processes in $\mathbf{\Psi}_0$ given by
\begin{align*}
	X^{(i,\varepsilon)}_t:=Y^{(i,\varepsilon)}_{\int_0^t X_s^{(i,\varepsilon)}ds}, \qquad \text{$t\geq0$, $i=1,2$,}
\end{align*}
where $Y^{(1,\varepsilon)}$ and $Y^{(2,\varepsilon)}$ are Poisson process with generating triplets $(b^{(1)},0,z\varepsilon^{-2}\delta_{\varepsilon})$ and $(b^{(2)},0,(z\varepsilon^{-2}+(s\varepsilon)^{-1})\delta_{\varepsilon})$, respectively.
Where
\begin{align*}
	b^{(1)}&:=-z\varepsilon^{-2}\int_{(0,\infty)}\left(\frac{x}{x^2+1}-x1_{\{x\leq 1\}}\right)\delta_{\varepsilon}(dx)=\frac{z}{\varepsilon^2}\frac{\varepsilon^3}{\varepsilon^2+1},\\ 
	b^{(2)}&:=-\frac{sz-\varepsilon}{s\varepsilon^2}\int_{(0,\infty)}\left(\frac{x}{x^2+1}-x1_{\{x\leq 1\}}\right)\delta_{\varepsilon}(dx)=\frac{sz-\varepsilon}{s\varepsilon^2}\frac{\varepsilon^3}{\varepsilon^2+1}. 
\end{align*}
Now, because
\begin{align*}
	\tilde{m}^{(1)}&:=z\varepsilon^{-2}\frac{x^2}{x^2+1}\delta_{\varepsilon}(dx)\to z\delta_0,\\
	\tilde{m}^{(2)}&:=\frac{sz-\varepsilon}{s\varepsilon^2}\frac{x^2}{x^2+1}\delta_{\varepsilon}(dx)\to z\delta_0,\qquad \text{weakly as $\varepsilon\to 0$.}
\end{align*}
we have by Proposition \ref{cont2} that for $i=1,2$, $X^{(i,\varepsilon)}\to X^{(i,0)}$ as $\varepsilon\to0$ weakly in $\mathbb{D}(\R_+,\R_+)$ where $X^{(i,0)}$ is the solution to
\begin{align*}
	X^{(i,0)}_t=x^{(i)}+\int_0^t\sqrt{2zX^{(i,0)}_s}dB_s^{(i)},\qquad t\geq0,
\end{align*}
and $B^{(i)}=\{B^{(i)}_t:t\geq0\}$ are independent Brownian motions. By Theorem \ref{theo_dual}, the dual process of the associated $\Lambda$-asymmetric frequency process has generator $\mathcal{Q}^{(z,0)}$, which satisfies
\begin{align*}
	\mathcal{Q}^{(z,0)}f_r(n)=\binom{n}{2}[r^{n-1}-r^n].
\end{align*}
Meanwhile, let us denote by $R^{(z,r,\varepsilon)}$ the $\Lambda$-asymmetric frequency process associated to the couple of CB processes $(X^{(1,\varepsilon)},X^{(2,\varepsilon)})$. Then, by \eqref{dual_exam} we have that the generator $\mathcal{Q}^{(z,\varepsilon)}$ of the dual process of $R^{(z,r,\varepsilon)}$ satisfies
\begin{align*}
	\mathcal{Q}^{(z,\varepsilon)}f_r(n)&=\frac{z^2}{\varepsilon^2}\sum_{k=2}^n\binom{n}{k}\left(\frac{z}{z+\varepsilon}\right)^{n-k}\left(\frac{\varepsilon}{z+\varepsilon}\right)^k[r^{n-k+1}-r^n]\nonumber\\  
	&+\left[z(z\varepsilon^{-2}+(s\varepsilon)^{-1})\left(1-\left(\frac{z}{z+\varepsilon}\right)^n\right)-\frac{z^2}{\varepsilon^2}\left(1-\left(\frac{z}{z+\varepsilon}\right)^n\right)\right][r^{n+1}-r^n]
\end{align*}
Therefore,
\begin{align*}
	\lim_{\varepsilon\to 0}\mathcal{Q}^{(z,\varepsilon)}f_r(n)=\binom{n}{2}[r^{n-1}-r^n]+\frac{n}{s}[r^{n+1}-r^n].
\end{align*}
The fact that for $r\in[0,1]$, and $n\geq 1$, $\lim_{\varepsilon\to 0}\mathcal{Q}^{(z,\varepsilon)}f_r(n)\not=\mathcal{Q}^{(z,0)}f_r(n)$ implies that the mapping that sends a couple of CB process in $\mathbf{\Psi}_0$ with different distributions to their associated $\Lambda$-asymmetric frequency process by the culling procedure and then by the duality relationship to the space of $\Lambda$-coalescent processes is in general not continuous. 

\section{Biological remarks}\label{biol_remarks}
By studying the relative frequency between two general CB processes, we have characterized the evolutionary forces that emerge from the differences in the reproduction mechanisms of two competing species, as follows:

\begin{enumerate}
	\item Selection: Classic selection is visible in the term $s$ of the generator of the dual process, given by \eqref{gendual}, in the form of branching. We distinguish three sources of selection coming from the difference of the reproduction mechanisms. The first is unsurprising: the difference between the drift terms $b^{(2)}-b^{(1)}$. The second is related to the difference between the diffusion terms and was first observed by Gillespie (see \cite{Gill73, Gill74}), having the form $2z^{-1}(c^{(2)}-c^{(1)})$. The third is new in the literature and comes from the difference of the terms associated with the compensation of the jump measures of the CB processes. 
	\item Frequency-dependent selection: The term $\beta_i$ in \eqref{gendual} corresponds to frequency-dependent selection, and to our knowledge is new in the literature. This frequency-dependent selection term is related to the jump measures of the competing CB processes.
	\item Frequency-dependent variance: As observed by Gillespie in \cite{Gill73, Gill74}, the difference between the diffusion terms $c^{(1)}$ and $c^{(2)}$ modifies the variance. To be precise Gillespie introduced the  Gillespie-Wright-Fisher diffusion, which he obtained as the relative frequency between two Feller processes and solves the following SDE:
	\begin{equation*}
		dX_t=X_t(1-X_t)\left[b^{(2)}-b^{(1)}+\frac{2}{z}(c^{(2)}-c^{(1)})\right]dt+\sqrt{\frac{2}{z}X_t(1-X_t)[c^{(2)}X_t+c^{(1)}(1-X_t)]}dB_t,
	\end{equation*}
	for $t>0$.
	
	In \cite{GMP}, a similar population dependent variance was obtained in the context of populations that require  different amount of resources to reproduce (efficiency). In \cite{GPP}, it was shown that the \textit{efficiency} term has a dual term, which is pairwise branching. We further generalize this by observing that there is an additional frequency dependent term modifying the variance associated with the jumps measures of the CB processes. These are the terms $\kappa_k$ in the transitions of the dual process in \eqref{gendual}.
	\item Coalescence: The terms $\overline{\lambda}^{(1)}_{i,i-j+1}$ in the generator of the dual process given in \eqref{gendual} are associated with coalescence. A novel characterisation of the $\Lambda$-coalescent arises naturally from this work: those that can be obtained as functionals of CB processes in the sense of being dual to an asymmetric frequency process. 
	\item Mutation: If one allows immigration, then one can obtain mutation. Mutation can be found in the terms $\alpha_i$ and $\overline{\mu}^{(1)}_{i,i-j}$ in \eqref{gendual}. If the immigration is discontinuous, modelling big immigration events, then one obtains coordinated mutation in the sense of \cite{GKT}. 
\end{enumerate}

A natural direction for future research is to study and understand the mechanisms behind the appearance of each of the new terms (e.g. the selection term arising from the jump measures or the frequency-dependent selection term) that were obtained due to the asymmetry in the dynamics of the CBI processes from which the $\Lambda$-asymmetric frequency process is constructed. 

\section*{Acknowledgements}
%
%
We want to thank the anonymous referees for the careful reading, constructive comments and suggestions, which
significantly improved the presentation and the readability of the paper. 

The second author was supported by the grant CONACYT CIENCIA B\'ASICA A1-S-14615.

\begin{appendix}
	\section{Proof of proposition \ref{infinitesimal_generator}}\label{App_1}
	Because $X^{(1)}$ and $X^{(2)}$ are semi-martingales, and $f$ is sufficiently smooth on $[0,1]\times[0, \infty)$, 
	we can use the change of variables/Meyer-It\^o's formula (cf.\ Theorems II.31 and II.32 of \cite{protter}) to deduce that 
	\begin{align}\label{gen_fs}
		&f\Bigg(\frac{X^{(1)}_{t\wedge\tau }}{X^{(1)}_{t\wedge\tau }+X^{(2)}_{t\wedge\tau }},X^{(1)}_{t\wedge\tau }+X^{(2)}_{t\wedge \tau}\Bigg)=f\left(\frac{x^{(1)}}{x^{(1)}+x^{(2)}},x^{(1)}+x^{(2)}\right)+M_{t\wedge\tau }\notag\\&+\int_0^{t\wedge\tau }\left[A^{(1)}(X^{(1)}_s,X^{(2)}_s)+A^{(2)}(X^{(1)}_s,X^{(2)}_s)+A^{(3)}(X^{(1)}_s,X^{(2)}_s)+A^{(4)}(X^{(1)}_s,X^{(2)}_s)\right]ds,
	\end{align}
	where 
	\begin{align*}
		&A^{(1)}(x,y):=-b^{(1)}x\partial_1f\left(\frac{x}{x+y},x+y\right)\frac{y}{(x+y)^2} -b^{(1)}x\partial_2f\left(\frac{x}{x+y},x+y\right)\notag\\&+c^{(1)}x\partial_{12}f\left(\frac{x}{x+y},x+y\right)\frac{y}{(x+y)^2}\notag\\
		&+c^{(1)}x\left(\partial_{11}f\left(\frac{x}{x+y},x+y\right)\frac{y^2}{(x+y)^4}-\partial_1f\left(\frac{x}{x+y},x+y\right)\frac{2y}{(x+y)^3}\right)\notag\\
		&+c^{(1)}x\left(\partial_{21}f\left(\frac{x}{x+y},x+y\right)\frac{y}{(x+y)^2}+\partial_{22}f\left(\frac{x}{x+y},x+y\right)\right),\notag
		\end{align*}
		\begin{align*}
		&A^{(2)}(x,y):=x\int_{(0,\infty)}\Bigg[f\left(\frac{x+u}{x+u+y},x+u+y\right)-f\left(\frac{x}{x+y},x+y\right)\notag\\
		&-u1_{(0,1)}(u)\left(\partial_1f\left(\frac{x}{x+y},x+y\right)\frac{y}{(x+y)^2}+\partial_2f\left(\frac{x}{x+y},x+y\right)\right)\Bigg]m^{(1)}(du)\notag\\
		&+\eta^{(1)}\partial_1f\left(\frac{x}{x+y},x+y\right)\frac{y}{(x+y)^2}+\eta^{(1)}\partial_2f\left(\frac{x}{x+y},x+y\right)\notag\\&+\int_{(0,\infty)}\left[f\left(\frac{x+u}{x+u+y},x+u+y\right)-f\left(\frac{x}{x+y},x+y\right)\right]\nu^{(1)}(du).
	\end{align*}
	Additionally,
	\begin{align*}
		&A^{(3)}(x,y)=b^{(2)}y\partial_1f\left(\frac{x}{x+y},x+y\right)\frac{x}{(x+y)^2}-b^{(2)}y\partial_2f\left(\frac{x}{x+y},x+y\right)\notag\\&-c^{(2)}y\partial_{12}f\left(\frac{x}{x+y},x+y\right)\frac{x}{(x+y)^2}\notag\\
		&+c^{(2)}y\left(\partial_{11}f\left(\frac{x}{x+y},x+y\right)\frac{x^2}{(x+y)^4}+\partial_1f\left(\frac{x}{x+y},x+y\right)\frac{2x}{(x+y)^3}\right)\notag\\
		&+c^{(2)}y\left(-\partial_{21}f\left(\frac{x}{x+y},x+y\right)\frac{x}{(x+y)^2}+\partial_{22}f\left(\frac{x}{x+y},x+y\right)\right),\notag\\
		&A^{(4)}(x,y)=y\int_{(0,\infty)}\Bigg[f\left(\frac{x}{x+u+y},x+u+y\right)-f\left(\frac{x}{x+y},x+y\right)\notag\\
		&-u1_{(0,1)}(u)\left(-\partial_1f\left(\frac{x}{x+y},x+y\right)\frac{x}{(x+y)^2}+\partial_2f\left(\frac{x}{x+y},x+y\right)\right)\Bigg]m^{(2)}(du)\notag\\
		&-\eta^{(2)}\partial_1f\left(\frac{x}{x+y},x+y\right)\frac{x}{(x+y)^2}+\eta^{(2)}\partial_2f\left(\frac{x}{x+y},x+y\right)\notag\\
		&+\int_{(0,\infty)}\left[f\left(\frac{x}{x+u+y},x+u+y\right)-f\left(\frac{x}{x+y},x+y\right)\right]\nu^{(2)}(du).
	\end{align*}
	and $M=\{M_t:t\geq0\}$ is a local martingale.
	
	Using \eqref{tpsp} together with \eqref{fp}, we can write \eqref{gen_fs} in terms of the process $(R,Z)$, as follows
	\begin{align}\label{gen_fs_2}
		f\left(R_{t\wedge\tau },Z_{t\wedge\tau }\right)&=f\left(r,z\right)+\int_0^{t\wedge\tau }\Big[B^{(1)}(R_s,Z_s)+B^{(2)}(R_s,Z_s)+B^{(3)}(R_s,Z_s)+B^{(4)}(R_s,Z_s)\Big]ds\notag\\&+M_{t\wedge\tau },
	\end{align}
	where 
	\begin{align*}
		&B^{(1)}(r,z):=-b^{(1)}\partial_1f\left(r,z\right)r(1-r)-b^{(1)}rz\partial_2f\left(r,z\right)+c^{(1)}\left(\partial_{21}f\left(r,z\right)r(1-r)+\partial_{22}f\left(r,z\right)rz\right)\notag\\
		&+c^{(1)}\partial_{12}f\left(r,z\right)(1-r)r+\frac{c^{(1)}}{z}\left(\partial_{11}f\left(r,z\right)r(1-r)^2-\partial_1f\left(r,z\right)2r(1-r)\right),\notag\\
		&B^{(2)(x,y)}:=rz\int_{(0,\infty)}\Bigg[f\left(r\left(1-\frac{u}{z+u}\right)+\frac{u}{z+u},z+u\right)-f\left(r,z\right)\notag\\
		&-u1_{(0,1)}(u)\left(\partial_1f\left(r,z\right)\frac{(1-r)}{z}+\partial_2f\left(r,z\right)\right)\Bigg]m^{(1)}(du)+\eta^{(1)}\partial_1f\left(r,z\right)\frac{(1-r)}{z}+\eta^{(1)}\partial_2f\left(r,z\right)\\&+\int_{(0,\infty)}\left[f\left(r\left(1-\frac{u}{z+u}\right)+\frac{u}{z+u},z+u\right)-f\left(r,z\right)\right]\nu^{(1)}(du),\notag
	\end{align*}
	In addition,
	\begin{align*}
		&B^{(3)}(r,z):=b^{(2)}\partial_1f\left(r,z\right)r(1-r)-b^{(2)}(1-r)z\partial_2f\left(r,z\right)-c^{(2)}\partial_{12}f\left(r,z\right)r(1-r)\notag\\
		&+\frac{c^{(2)}}{z}\left(\partial_{11}f\left(r,z\right)r^2(1-r)+\partial_1f\left(r,z\right)2r(1-r)\right)+c^{(2)}\left(-\partial_{21}f\left(r,z\right)r(1-r)+(1-r)z\partial_{22}f\left(r,z\right)\right).
	\end{align*}
	\begin{align*}
		&B^{(4)}(x,y):=(1-r)z\int_{(0,\infty)}\Bigg[f\left(r\left(1-\frac{u}{z+u}\right),z+u\right)-f\left(r,z\right)\\&-u1_{(0,1)}(u)\left(-\partial_1f\left(r,z\right)\frac{r}{z}+\partial_2f\left(r,z\right)\right)\Bigg]m^{(2)}(du)\notag\\
		&-\eta^{(2)}\partial_1f\left(r,z\right)\frac{r}{z}+\eta^{(2)}\partial_2f\left(r,z\right)+\int_{(0,\infty)}\left[f\left(r\left(1-\frac{u}{z+u}\right),z+u\right)-f\left(r,z\right)\right]\nu^{(2)}(du).
	\end{align*}
	Hence, 
	noting that for $(r,z)\in[0,1]\times[0,\infty)$
	\begin{align*}
		\mathcal{L}f(r,z)=B^{(1)}(r,z)+B^{(2)}(r,z)+B^{(3)}(r,z)+B^{(4)}(r,z),
	\end{align*}
	we obtain the result.
	\section{Proof of Proposition \ref{exi_uni_sde}}\label{App_2}
	\textit{Step 1.-} First we will prove that any solution $R^{(z,r)}$ to \eqref{sde_p_cropped} satisfies that $R^{(z,r)}_t\in[0,1]$ for all $t\geq0$ $\mathbb{P}$-a.s. To this end, let us denote for $x\geq 0$
	\begin{align*}
		b(x)&:=(b^{(2)}-b^{(1)})x(1-x)1_{\{x\in[0,1]\}}+\frac{2}{z}(c^{(2)}-c^{(1)})x(1-x)1_{\{x\in[0,1]\}}+\eta^{(1)}\frac{(1-x)}{z}-\eta^{(2)}\frac{x}{z}\notag\\
		&+x(1-x)1_{\{x\in[0,1]\}}\int_{(0,1)}\frac{w^2}{z+w}m^{(2)}(dw)-x(1-x)1_{\{x\in[0,1]\}}\int_{(0,1)}\frac{w^2}{z+w}m^{(1)}(dw),\notag\\
		\sigma(x)&:=\sqrt{\frac{2}{z}x(1-x)(c^{(1)}(1-x)+c^{(2)}x)}1_{\{x\in[0,1]\}}. 
	\end{align*}
	We note that the following conditions are satisfied:
	\begin{itemize}
		\item[(i)] $\sigma(x)=0$ for all $x\in\R\backslash[0,1]$.
		\item[(ii)] For $x>1$, $$b(x)=\eta^{(1)}\frac{(1-x)}{z}-\eta^{(2)}\frac{x}{z}\leq 0,$$
		and for $x<0$, $$b(x)=\eta^{(1)}\frac{(1-x)}{z}-\eta^{(2)}\frac{x}{z}\geq 0.$$
		\item[(iii)] For $(w,v)\in(0,\infty)^2$ 
		\[
		0\leq h^{(z)}(x,w,v)+x=x-\frac{w}{z+w}x1_{\{v\leq (1-x)z\}}\leq1,\qquad \text{for $x\in[0,1]$,}
		\]
		and $h^{(z)}(x,w,v)=0$ for $x\in\R\backslash[0,1]$.
		\item[(iv)] For $w\in(0,\infty)$ observe that
		\[
		0\leq \tilde{h}^{(z)}(x,w)+x=x-\frac{w}{z+w}x \leq1,\qquad \text{for $x\in[0,1]$.}
		\]
		and $h^{(z)}(x,w)=0$ for $x\in\R\backslash[0,1]$.
		\item[(v)] For $(w,v)\in(0,\infty)^2$
		\[
		0\leq g^{(z)}(x,w,v)+x=x+\frac{w}{z+w}(1-x)1_{\{v\leq xz\}}\leq1,\qquad \text{for $x\in[0,1]$,}
		\]
		and $g^{(z)}(x,w,v)=0$ for $x\in\R\backslash[0,1]$.
		\item[(vi)] For $w\in(0,\infty)$ observe that
		\[
		0\leq \tilde{g}^{(z)}(x,w)+x=x+\frac{w}{z+w}(1-x) \leq1,\qquad \text{for $x\in[0,1]$.}
		\]
		and $\tilde{g}^{(z)}(x,w)=0$ for $x\in\R\backslash[0,1]$.
	\end{itemize}
	
	Then by a modification of Proposition 2.1 in \cite{LF}, extending the proof to the case of two boundaries (see also Corollary 6.2 in [32]), we have that $\mathbb{P}(R^{(z,r)}_t\in[0,1] \ \text{for all $t\geq0$})=1$.
	
	\textit{Step 2.-} To prove the existence of a strong solution to \eqref{sde_p_cropped}, we first obtain the following estimations:
	
	(i) For $x,y\in[0,1]$ we obtain
	\begin{align}\label{bound_drift_1}
		|b(x)&-b(y)|\notag\\&\leq \Bigg(|b^{(2)}-b^{(1)}|+\frac{2}{z}|c^{(2)}-c^{(1)}|+\frac{\eta^{(1)}}{z}+\frac{\eta^{(2)}}{z}+\int_{(0,1)}\frac{w^2}{z+w}(m^{(1)}+m^{(2)})(dw)\Bigg)|x-y|.
	\end{align}
	Additionally, for $x,y\in[0,1]$
	\begin{align}\label{bound_drift_2}
		\int_{(0,\infty)}|\tilde{g}^{(z)}(x,w)-\tilde{g}^{(z)}(y,w)|\nu^{(1)}(dw)&=\int_{(0,\infty)}\frac{w}{z+w}\nu^{(1)}(dw)|x-y|,\notag\\
		\int_{(0,\infty)}|\tilde{h}^{(z)}(x,w)-\tilde{h}^{(z)}(y,w)|\nu^{(2)}(dw)&=\int_{(0,\infty)}\frac{w}{z+w}\nu^{(2)}(dw)|x-y|.	
	\end{align}
	Now we note that for $x,y\in[0,1]$ and $(w,v)\in(0,\infty)^2$
	\begin{align}\label{g_h_lipschitz}
		|g^{(z)}(x,w,v)-g^{(z)}(y,w,v)|&\leq \frac{w}{z+w}\big(|x-y|1_{\{v\leq z(x\wedge y)\}}+(1-x)1_{\{yz<v\leq xz\}}\notag\\&+(1-y)1_{\{xz<v\leq yz\}}\big),\notag\\
		|h^{(z)}(x,w,v)-h^{(z)}(y,w,v)|&\leq \frac{w}{z+w}\big(|x-y|1_{\{v\leq z((1-x)\wedge (1-y))\}}+x1_{\{(1-y)z<v\leq (1-x)z\}}\notag\\&+y1_{\{(1-x)z<v\leq (1-y)z\}}\big).
	\end{align}
	Therefore, for $x,y\in[0,1]$
	\begin{align}\label{bound_drift_3}
		\int_{(0,\infty)}\int_{[1,\infty)}|g^{(z)}(x,w,v)-g^{(z)}(y,w,v)|m^{(1)}(dw)dv\leq 3z\int_{[1,\infty)}\frac{w}{z+w}m^{(1)}(dw)|x-y|,\notag\\
		\int_{(0,\infty)}\int_{[1,\infty)}|h^{(z)}(x,w,v)-h^{(z)}(y,w,v)|m^{(2)}(dw)dv\leq 3z\int_{[1,\infty)}\frac{w}{z+w}m^{(2)}(dw)|x-y|.
	\end{align}
	Hence, by using \eqref{bound_drift_1}, \eqref{bound_drift_2}, and \eqref{bound_drift_3} we have that for every $x,y\in[0,1]$ there exists $K_1(z)>0$ such that
	\begin{align}\label{ex_un_1}
		|b(x)-b(y)|&+\int_{(0,\infty)}|\tilde{g}^{(z)}(x,w)-\tilde{g}^{(z)}(y,w)|\nu^{(1)}(dw)+\int_{(0,\infty)}|\tilde{h}^{(z)}(x,w)-\tilde{h}^{(z)}(y,w)|\nu^{(2)}(dw)\notag\\
		&+\int_{(0,\infty)}\int_{[1,\infty)}|g^{(z)}(x,w,v)-g^{(z)}(y,w,v)|m^{(1)}(dw)dv\notag\\&+\int_{(0,\infty)}\int_{[1,\infty)}|h^{(z)}(x,w,v)-h^{(z)}(y,w,v)|m^{(2)}(dw)dv\leq K_1(z)|x-y|.
	\end{align}
	(ii) Meanwhile, for $x,y\in[0,1]$
	\begin{align*}
		|\sigma(x)-\sigma(y)|^2\leq\frac{6(c^{(1)}+c^{(2)})}{z}|x-y|.
	\end{align*}
	In addition, by \eqref{g_h_lipschitz}, we have for $x,y\in[0,1]$
	\begin{align}\label{ex_un_3}
		\int_{(0,\infty)}\int_{(0,1)}|g^{(z)}(x,w,v)-g^{(z)}(y,w,v)|^2m^{(1)}(dw)dv&\leq3z \int_{(0,\infty)}\frac{w^2}{(z+w)^2}m^{(1)}(dw)|x-y|,\notag\\
		\int_{(0,\infty)}\int_{(0,1)}|h^{(z)}(x,w,v)-h^{(z)}(y,w,v)|^2m^{(2)}(dw)dv&\leq3z \int_{(0,\infty)}\frac{w^2}{(z+w)^2}m^{(2)}(dw)|x-y|.
	\end{align}
	The previous identities imply that there exists $K_2(z)>0$
	\begin{align}\label{ex_un_2}
		|\sigma(x)-\sigma(y)|^2&+\int_{(0,\infty)}\int_{(0,1)}|g^{(z)}(x,w,v)-g^{(z)}(y,w,v)|^2m^{(1)}(dw)dv\notag\\
		&+\int_{(0,\infty)}\int_{(0,1)}|h^{(z)}(x,w,v)-h^{(z)}(y,w,v)|^2m^{(2)}(dw)dv\leq K_2(z)|x-y|.
	\end{align}
	(iii) We note that for $x\in[0,1]$, the mapping
	\begin{align*}
		x\mapsto x+g^{(z)}(x,w,v)=x+\frac{w}{z+w}(1-x)1_{\{v\leq xz\}}=
		\begin{cases} x &\mbox{if } v> xz, \\
			\displaystyle x\left(1-\frac{w}{z+w}\right)+\frac{w}{z+w} & \mbox{if } v\leq xz. \end{cases} 
	\end{align*}
	is non-decreasing for fixed $(w,v)\in(0,\infty)^2$. In addition, the mapping
	\begin{align*}
		x\mapsto x+h^{(z)}(x,w,v)=x-\frac{w}{z+w}x1_{\{v\leq (1-x)z\}}=
		\begin{cases} x &\mbox{if } v> (1-x)z, \\
			\displaystyle x\left(1-\frac{w}{z+w}\right) & \mbox{if } v\leq (1-x)z, \end{cases} 
	\end{align*}
	is also non-decreasing for fixed $(w,v)\in(0,\infty)^2$.
	
	(iv) We note that
	\begin{align*}
		\int_{(0,\infty)}\int_{(0,\infty)}&|g^{(z)}(x,w,v)|^2m^{(1)}(dw)dv+\int_{(0,\infty)}\int_{(0,\infty)}|h^{(z)}(x,w,v)|^2m^{(2)}(dw)dv\\
		&+\int_{(0,\infty)}|\tilde{g}^{(z)}(x,w,v)|^2\nu^{(1)}(dw)+\int_{(0,\infty)}|\tilde{h}^{(z)}(x,w,v)|^2\nu^{(2)}(dw)\notag\\
		&\leq\sum_{i=1}^2\left(z\int_{(0,\infty)}\frac{w^2}{(z+w)^2}m^{(i)}(dw)+\int_{(0,\infty)}\frac{w^2}{(z+w)^2}\nu^{(i)}(dw)\right).
	\end{align*}
	Hence, by proceeding as in \eqref{ex_un_1} and \eqref{ex_un_2}, we can find a constant $K(z)>0$ such that for every $x\in\R$
	\begin{align}\label{ex_un_new}
		|\sigma(x)&|^2+|b(x)|^2+\int_{(0,\infty)}\int_{(0,\infty)}|g^{(z)}(x,w,v)|^2m^{(1)}(dw)dv\notag\\&+\int_{(0,\infty)}\int_{(0,\infty)}|h^{(z)}(x,w,v)|^2m^{(2)}(dw)dv\notag\\
		&+\int_{(0,\infty)}|\tilde{g}^{(z)}(x,w,v)|^2\nu^{(1)}(dw)+\int_{(0,\infty)}|\tilde{h}^{(z)}(x,w,v)|^2\nu^{(2)}(dw)\notag\\
		&+\Bigg[\int_{(0,\infty)}\int_{[1,\infty)}|g^{(z)}(x,w,v)|m^{(1)}(dw)dv+\int_{(0,\infty)}\int_{[1,\infty)}|h^{(z)}(x,w,v)|m^{(2)}(dw)dv\notag\\
		&+\int_{(0,\infty)}|\tilde{g}^{(z)}(x,w,v)|\nu^{(1)}(dw)+\int_{(0,\infty)}|\tilde{h}^{(z)}(x,w,v)|\nu^{(2)}(dw)dv\Bigg]^2\leq K(z)(1+|x|^2).
	\end{align}
	The fact that the mappings $x\mapsto x+g^{(z)}(x,w,v)$ and $x\mapsto x+h^{(z)}(x,w,v)$ are non-decreasing for $(x,w,v)\in[0,1]\times(0,\infty)^2$, allows us to apply Lemma 3.1 in \cite{LP}.
	Hence, the inequalities \eqref{ex_un_1}, \eqref{ex_un_2}, and \eqref{ex_un_new} 
	imply, using a slight modification of Theorem 5.1 in \cite{LP}, that there exists a unique strong solution to \eqref{sde_p_cropped}.
	
	\textit{Step 3.-} To prove the last assertion in the statement, we obtain by the proof of Theorem 3.2 \cite{LP} together with the estimates \eqref{ex_un_1}, \eqref{ex_un_2}, and \eqref{ex_un_new}, and the fact that the mappings $x\mapsto x+g^{(z)}(x,w,v)$ and $x\mapsto x+h^{(z)}(x,w,v)$ are non-decreasing for $(x,w,v)\in[0,1]\times(0,\infty)^2$, that for $r,\overline{r}\in[0,1]$
	\begin{align*}
		\E\left[|R^{(r,z)}_t-R^{(\overline{r},z)}_t|\right]\leq|r-\overline{r}|+K_1(z)\int_0^t\E\left[|R^{(r,z)}_s-R^{(\overline{r},z)}_s|\right]ds, \ \text{$t\geq0$.}
	\end{align*}
	Hence, by an application of Gronwall's inequality, we obtain \eqref{estimate}.
	\section{Proof of Theorem \ref{large_pop}}\label{App_3}
	(i) First we note that $R^{(\infty,r)}$ is a solution to the following ordinary differential equation
	\begin{align}\label{ode_p_cropped}
		&dR^{(\infty,r)}_t=R^{(\infty,r)}_t(1-R^{(\infty,r)}_t)\left(b^{(2)}-\int_{[1,\infty)}wm^{(2)}(dw)-b^{(1)}+\int_{[1,\infty)}wm^{(1)}(dw)\right)dt,\quad  t> 0,
	\end{align}
	with $R^{(\infty,r)}_0=r$.
	
	Hence, by using \eqref{sde_p_cropped} together with \eqref{ode_p_cropped} we obtain that
	\begin{align}\label{bound_conv_z_0}
		R^{(z,r)}_t-R^{(\infty,r)}_t=A^{(1,z)}_t+A^{(2,z)}_t+A^{(3,z)}_t,\qquad t\geq 0,
	\end{align}
	where for $t\geq0$
	\begin{align}\label{def_A}
		&A^{(1,z)}_t:=\int_0^t\sqrt{\frac{2}{z}R^{(z,r)}_{s-}(1-R^{(z,r)}_{s-})[c^{(1)}(1-R^{(z,r)}_{s-})+c^{(2)}R^{(z,r)}_{s-}]}dB_s\notag\\
		&+\int_0^t\int_{(0,\infty)^2}g^{(z)}(R^{(z,r)}_{s-},w,v)\tilde{N}_1(ds,dw,dv)+\int_0^t\int_{(0,\infty)^2}h^{(z)}(R^{(z,r)}_{s-},w,v)\tilde{N}_2(ds,dw,dv)\notag\\
		&+\int_0^t\int_{(0,\infty)}\tilde{g}^{(z)}(R^{(z,r)}_{s-},w)\tilde{N}_3(ds,dw)+\int_0^t\int_{(0,\infty)}\tilde{h}^{(z)}(R^{(z,r)}_{s-},w)\tilde{N}_4(ds,dw)\notag\\
		&A^{(2,z)}_t:=\int_0^t\Bigg[R^{(z,r)}_s(1-R^{(z,r)}_s)\Bigg(\frac{2}{z}(c^{(2)}-c^{(1)})+\int_{(0,1)}\frac{w^2}{w+z}m^{(2)}(dw)-\int_{(0,1)}\frac{w^2}{w+z}m^{(1)}(dw)\Bigg)\Bigg]ds\notag\\
		&+\int_0^t\Bigg[\eta^{(1)}\frac{(1-R^{(z,r)}_s)}{z}-\eta^{(2)}\frac{R^{(z,r)}_s}{z}\Bigg]ds+\int_0^t\int_{(0,\infty)}\tilde{g}^{(z)}(R^{(z,r)}_s,w)\nu^{(1)}(dw)ds\notag\\&+\int_0^t\int_{(0,\infty)}\tilde{h}^{(z)}(R^{(z,r)}_s,w)\nu^{(2)}(dw)ds,
	\end{align}
	and
	\begin{align}\label{def_A_3}
		&A^{(3,z)}_t:=\int_0^t\int_{[1,\infty)}R^{(z,r)}_s(1-R^{(z,r)}_s)\frac{wz}{w+z}m^{(1)}(dw)ds-\int_0^t\int_{[1,\infty)}R^{(\infty,r)}_s(1-R^{(\infty,r)}_s)wm^{(1)}(dw)ds\notag\\
		&-\int_0^t\int_{[1,\infty)}R^{(z,r)}_s(1-R^{(z,r)}_s)\frac{wz}{w+z}m^{(2)}(dw)ds+\int_0^t\int_{[1,\infty)}R^{(\infty,r)}_s(1-R^{(\infty,r)}_s)wm^{(2)}(dw)ds\notag\\
		&+\int_0^t(b^{(2)}-b^{(1)})R^{(z,r)}_s(1-R^{(z,r)}_s)ds-\int_0^t(b^{(2)}-b^{(1)})R^{(\infty,r)}_s(1-R^{(\infty,r)}_s)ds.
	\end{align}
	(ii) We will start by obtaining some estimations for the term $A^{(1,z)}$, so using Doob's inequality we have  for $t\in[0,T]$ 
	\begin{align}\label{est_mart_cont}
		\E\Bigg[\Bigg(\sup_{u\leq t}&\int_0^u\sqrt{\frac{2}{z}R^{(z,r)}_{s-}(1-R^{(z,r)}_{s-})[c^{(1)}(1-R^{(z,r)}_{s-})+c^{(2)}R^{(z,r)}_{s-}]}dB_s\Bigg)^2\Bigg]\notag\\&\leq\frac{C_1}{z}\E\left[\int_0^T\left(R^{(z,r)}_s(1-R^{(z,r)}_s)[c^{(1)}(1-R^{(z,r)}_s)+c^{(2)}R^{(z,r)}_s]\right)ds\right]\leq \frac{C_1}{z}(c^{(1)}+c^{(2)})T,
	\end{align}
	for some constant $C_1>0$.
	Next, by another application of Doob's inequality, there exists a constant $C_2>0$ such that for $t\in[0,T]$ 
	\begin{align}\label{est_mart_discont_1}
		\E\Bigg[\Bigg(\sup_{u\leq t}\int_0^u\int_{(0,\infty)^2}&g^{(z)}(R^{(z,r)}_{s-},w,v)\tilde{N}_1(ds,dw,dv)\Bigg)^2\Bigg]\notag\\
		&\leq C_2\E\left[\int_0^T\int_{(0,\infty)}\left(\frac{w}{z+w}(1-R^{(z,r)}_s)\right)^2zR^{(z,r)}_sm^{(1)}(dw)ds\right]\notag\\&\leq C_2\left(\frac{1}{z}\int_{(0,1)}w^2m^{(1)}(dw)+\int_{[1,\infty)}\frac{w^2z}{(w+z)^2}m^{(1)}(dw)\right)T.
	\end{align}
	Proceeding as in \eqref{est_mart_discont_1}, we have for $t\in[0,T]$ 
	\begin{align}\label{est_mart_discont_2}
		\E\Bigg[\Bigg(\sup_{u\leq t}\int_0^u\int_{(0,\infty)^2}&h^{(z)}(R^{(z,r)}_{s-},w,v)\tilde{N}_2(ds,dw,dv)\Bigg)^2\Bigg]
		\notag\\&\leq C_3\left(\frac{1}{z}\int_{(0,1)}w^2m^{(2)}(dw)+\int_{[1,\infty)}\frac{w^2z}{(w+z)^2}m^{(2)}(dw)\right)T,
	\end{align}
	and
	\begin{align}\label{est_mart_discont_3}
		\E\Bigg[\Bigg(\sup_{u\leq t}\int_0^u\int_{(0,\infty)}&\tilde{g}^{(z)}(R^{(z,r)}_{s-},w)\tilde{N}_3(ds,dw)\Bigg)^2\Bigg]\notag\\&+\E\Bigg[\Bigg(\sup_{u\leq t}\int_0^u\int_{(0,\infty)}\tilde{h}^{(z)}(R^{(z,r)}_{s-},w)\tilde{N}_4(ds,dw)\Bigg)^2\Bigg]\notag\\
		&\leq C_3\E\left[\int_0^T\int_{(0,\infty)}\frac{w^2}{(w+z)^2}(1-R^{(z,r)}_s)^2\nu^{(1)}(dw)ds\right]\notag\\&+C_3\E\left[\int_0^T\int_{(0,\infty)}\frac{w^2}{(w+z)^2}(R^{(z,r)}_s)^2\nu^{(2)}(dw)ds\right]\notag\\
		&\leq C_3T\left(\int_{(0,\infty)}\frac{w^2}{(w+z)^2}\nu^{(1)}(dw)+\int_{(0,\infty)}\frac{w^2}{(w+z)^2}\nu^{(2)}(dw)\right),
	\end{align}
	for $t\in[0,T]$ and some constant $C_3>0$.
	
	Therefore, by using inequalities \eqref{est_mart_cont}, \eqref{est_mart_discont_1}, \eqref{est_mart_discont_2}, and \eqref{est_mart_discont_3} together with the fact that $\int_{[1,\infty)}wm^{(i)}(dw)<\infty$ for $i=1,2$, we can find a constant $K_1(T,z)>0$ such that
	\begin{align}\label{bound_conv_z_1}
		\E\left[\left(\sup_{u\leq t}A^{(1,z)}_u\right)^2\right]\leq K_1(T,z),\qquad t\in[0,T],
	\end{align}
	such that $\lim_{z\to\infty}K_1(T,z)=0$.
	
	(ii) Now, for the term $A^{(2,z)}$ we note that for $t\in[0,T]$ 
	\begin{align}\label{bv_bounds_1}
		\E\Bigg[\Bigg(\sup_{u\leq t}\int_0^u\Bigg[&R^{(z,r)}_s(1-R^{(z,r)}_s)\Bigg(\frac{2}{z}(c^{(2)}-c^{(1)})+\int_{(0,1)}\frac{w^2}{w+z}m^{(2)}(dw)-\int_{(0,1)}\frac{w^2}{w+z}m^{(1)}(dw)\Bigg)\notag\\
		&+\eta^{(1)}\frac{(1-R^{(z,r)}_s)}{z}ds-\eta^{(2)}\frac{R^{(z,r)}_s}{z}\Bigg]ds\Bigg)^2\Bigg]\notag\\
		&\leq \frac{T^2}{z^2}\Bigg(2|c^{(2)}-c^{(1)}|+\int_{(0,1)}w^2m^{(2)}(dw)+\int_{(0,1)}w^2m^{(1)}(dw)+\eta^{(1)}+\eta^{(2)}\Bigg)^2,
	\end{align}
	and
	\begin{align}\label{bv_bounds_2}
		\E\Bigg[\sup_{u\leq t}\Bigg(&\int_0^u\int_{(0,\infty)}\tilde{g}^{(z)}(R^{(z,r)}_s,w)\nu^{(1)}(dw)ds\blue{+}\int_0^u\int_{(0,\infty)}\tilde{h}^{(z)}(R^{(z,r)}_s,w)\nu^{(2)}(dw)ds\Bigg)^2\Bigg]\notag\\
		&\leq T^2\Bigg(\int_{(0,\infty)}\frac{w}{w+z}\nu^{(1)}(dw)+\int_{(0,\infty)}\frac{w}{w+z}\nu^{(2)}(dw)\Bigg)^2, \qquad   t\in[0,T].
	\end{align}
	Hence, by \eqref{bv_bounds_1} and \eqref{bv_bounds_2} there exists a constant $K_2(T,z)>0$ such that
	\begin{align}\label{bound_conv_z_2}
		\E\left[\left(\sup_{u\leq t}A^{(2,z)}_u\right)^2\right]\leq K_2(T,z),\qquad t\in[0,T],
	\end{align}
	such that $\lim_{z\to\infty}K_2(T,z)=0$.
	
	(ii) For the term $A^{(3,z)}$, we have for $t\in[0,T]$
	\begin{align}\label{lip_bounds_1}
		\E\Bigg[\sum_{i=1}^2\Bigg(\sup_{u\leq t}\Bigg|\int_0^u\int_{(1,\infty)}&R^{(z,r)}_s(1-R^{(z,r)}_s)\frac{wz}{w+z}m^{(i)}(dw)ds\notag\\&-\int_0^u\int_{(1,\infty)}R^{(\infty,r)}_s(1-R^{(\infty,r)}_s)wm^{(i)}(dw)ds\Bigg|\Bigg)^2\Bigg]\notag\\
		&\leq 9\sum_{i=1}^2\left(\int_{[1,\infty)}wm^{(i)}(dw)\right)^2T\E\left[\int_0^t\sup_{u\leq s}|R^{(z,r)}_u-R^{(\infty,r)}_u|^2du\right]\notag\\&+9T^2\sum_{i=1}^2\left(\int_{[1,\infty)}\frac{w^2}{w+z}m^{(i)}(dw)\right)^2,
	\end{align}
	and,
	\begin{align}\label{lip_bounds_2}
		\E\Bigg[\Bigg(\sup_{u\leq t}\Bigg|&\int_0^t(b^{(2)}-b^{(1)})R^{(z,r)}_s(1-R^{(z,r)}_s)ds-\int_0^t(b^{(2)}-b^{(1)})R^{(\infty,r)}_s(1-R^{(\infty,r)}_s)ds\Bigg|\Bigg)^2\Bigg]\notag\\
		&\leq 9(b^{(2)}-b^{(1)})^2T\E\left[\int_0^t\sup_{u\leq s}|R^{(z,r)}_u-R^{(\infty,r)}_u|^2du\right], \qquad t\in[0,T].
	\end{align}
	Therefore, by using \eqref{lip_bounds_1} and \eqref{lip_bounds_2}, we can find constants  $K_3(T,z),K_4(T)>0$ such that
	\begin{align}\label{bound_conv_z_3}
		\E\left[\left(\sup_{u\leq t}A^{(3,z)}_u\right)^2\right]\leq K_3(T,z)+K_4(T)\E\left[\int_0^t\sup_{u\leq s}|R^{(z,r)}_u-R^{(\infty,r)}_u|^2du\right],\qquad t\in[0,T],
	\end{align}
	such that $\lim_{z\to\infty}K_3(T,z)=0$.
	
	(iv) By \eqref{bound_conv_z_1}, \eqref{bound_conv_z_2} and \eqref{bound_conv_z_3} together with \eqref{bound_conv_z_0}, we have for $t\in[0,T]$,
	\begin{align*}
		\E\left[\sup_{u\leq t}|R^{(z,r)}_u-R^{(\infty,r)}_u|^2\right]\leq \sum_{i=1}^3K_i(T,z)+K_4(T)\int_0^t\E\left[\sup_{u\leq s}|R^{(z,r)}_u-R^{(\infty,r)}_u|^2\right]du.
	\end{align*}
	Hence, by an application of Gronwall's inequality, we obtain that for $T>0$ 
	\begin{align}\label{gron}
		\E\left[\sup_{u\leq T}|R^{(z,r)}_u-R^{(\infty,r)}_u|^2\right]\leq \sum_{i=1}^3K_i(T,z)e^{K_4(T)T}\to 0,\qquad\text{as $z\to\infty$.}
	\end{align} 
	\section{Proof of Lemma \ref{iden-limit}}\label{App_4}
	We will characterize the finite-dimensional distributions of the weak limit $Y^{(\infty)}$ by means of its characteristic function. To this end, consider $0\leq t_0<t_1<\dots<t_n\leq T$ and $a_i\in \R$ for $i=1,\dots,n$, and denote
	\begin{align*}
		\Phi(\lambda,z)=\E\left[e^{i\lambda \sum_{i=1}^na_i(X^{(z)}(t_i)-X^{(z)}(t_{i-1}))}\right],\qquad  \lambda\in\R.
	\end{align*}
	Using the fact that $B$, $\tilde{N}_1(dt,du,dv)$, and $\tilde{N}_2(dt,du,dv)$ are independent, we obtain for $\lambda\in\R$
	\begin{align}\label{char_0}
		\Phi(\lambda,z)=\Phi_1(\lambda)\prod_{i=2}^3\Phi_i(\lambda,z),
	\end{align}
	where
	\begin{align*}
		\Phi_1(\lambda)&=\E\left[\exp\left\{i\lambda \sum_{i=1}^na_i\int_{t_{i-1}}^{t_i}e^{U_t}\sqrt{2R^{(\infty,r)}_t(1-R^{(\infty,r)}_t)[c^{(1)}(1-R^{(\infty,r)}_t)+c^{(2)}R^{(\infty,r)}_t]}dB_t\right\}\right],\\
		\Phi_2(\lambda,z)&=\E\left[\exp\left\{i\lambda \sum_{i=1}^na_i\sqrt{z}\int_{t_{i-1}}^{t_i}\int_{(0,\infty)^2}e^{U_t}g^{(z)}(R^{(\infty,r)}_t,w,v)\tilde{N}_1(dt,dw,dv)\right\}\right],\\
		\Phi_3(\lambda,z)&=\E\left[\exp\left\{i\lambda \sum_{i=1}^na_i\sqrt{z}\int_{t_{i-1}}^{t_i}\int_{(0,\infty)^2}e^{U_t}h^{(z)}(R^{(\infty,r)}_t,w,v)\tilde{N}_2(dt,dw,dv)\right\}\right].
	\end{align*}
	Noting that the process $R^{(\infty,r)}$ is deterministic, we obtain
	\begin{align}\label{char_1}
		\Phi_1(\lambda)&=\prod_{i=1}^n\E\left[\exp\left\{i\lambda a_i\int_{t_{i-1}}^{t_i}e^{U_t}\sqrt{2R^{(\infty,r)}_t(1-R^{(\infty,r)}_t)[c^{(1)}(1-R^{(\infty,r)}_t)+c^{(2)}R^{(\infty,r)}_t]}dB_t\right\}\right]\notag\\
		&=\prod_{i=1}^n\exp\left\{-\frac{(\lambda a_i)^2}{2}\int_{t_{i-1}}^{t_i}e^{2U_t}[2R^{(\infty,r)}_t(1-R^{(\infty,r)}_t)[c^{(1)}(1-R^{(\infty,r)}_t)+c^{(2)}R^{(\infty,r)}_t]dt\right\}.
	\end{align}
	The exponential formula for Poisson random measures allows us to write
	\begin{align*}
		\Phi_2(\lambda,z)&=\prod_{i=1}^n\E\left[\exp\left\{i\lambda a_i\sqrt{z}\int_{t_{i-1}}^{t_i}\int_{(0,\infty)^2}e^{U_s}g^{(z)}(R^{(\infty,r)}_s,w,v)\tilde{N}_1(ds,dw,dv)\right\}\right]\notag\\
		&=\prod_{i=1}^n\exp\Bigg\{-\int_{t_{i-1}}^{t_i}\int_{(0,\infty)}\Bigg(1-\exp\left\{i\lambda a_ie^{U_s}(1-R^{(\infty,r)}_s)\frac{w\sqrt{z}}{z+w}\right\}\notag\\&\hspace{ 5cm}+i\lambda a_ie^{U_s}(1-R^{(\infty,r)}_s)\frac{w\sqrt{z}}{z+w}\Bigg)zR^{(\infty,r)}_sm^{(1)}(dw)ds\Bigg\}.
	\end{align*}
	Hence, since $\int_{(0,\infty)}w^2m^{(1)}(dw)<\infty$, we have
	\begin{align}\label{char_2}
		\lim_{z\to\infty}\Phi_2(\lambda,z)=\prod_{i=1}^n\exp\left\{-\frac{(\lambda a_i)^2}{2}\int_{(0,\infty)}w^2m^{(1)}(dw)\int_{t_{i-1}}^{t_i}e^{2U_s}(1-R^{(\infty,r)}_s)^2R^{(\infty,r)}_sds\right\}.
	\end{align}
	Proceeding as in \eqref{char_2}
	\begin{align*}
		&\Phi_3(\lambda,z)=\prod_{i=1}^n\E\left[\exp\left\{i\lambda a_i\sqrt{z}\int_{t_{i-1}}^{t_i}\int_{(0,\infty)^2}e^{U_s}h^{(z)}(R^{(\infty,r)}_s,w,v)\tilde{N}_2(ds,dw,dv)\right\}\right]\notag\\
		&=\prod_{i=1}^n\exp\Bigg\{-\int_{t_{i-1}}^{t_i}\int_{(0,\infty)}\left(1-\exp\Bigg\{-i\lambda a_ie^{U_s}R^{(\infty,r)}_s\frac{w\sqrt{z}}{z+w}\right\}\notag\\&\hspace{ 5cm}-i\lambda a_ie^{U_s}R^{(\infty,r)}_s\frac{w\sqrt{z}}{z+w}\Bigg)z(1-R^{(\infty,r)}_s)m^{(2)}(dw)ds\Bigg\}.
	\end{align*}
	Now, because $\int_{(0,\infty)}z^2m^{(2)}(dz)<\infty$, we obtain
	\begin{align}\label{char_3}
		\lim_{z\to\infty}\Phi_3(\lambda,z)=\prod_{i=1}^n\exp\left\{-\frac{(\lambda a_i)^2}{2}\int_{(0,\infty)}w^2m^{(2)}(dw)\int_{t_{i-1}}^{t_i}e^{2U_s}(1-R^{(\infty,r)}_s)(R^{(\infty,r)}_s)^2ds\right\}.
	\end{align}
	Using \eqref{char_1}, \eqref{char_2}, and \eqref{char_3} in \eqref{char_0}, we obtain that
	\begin{align}\label{char_5}
		&\E\Big[e^{i\lambda \sum_{i=1}^na_i(Y_{t_i}^{(\infty)}-Y_{t_{i-1}}^{(\infty)})}\Big]\notag\\&=\lim_{z\to\infty}\Phi(z,\lambda)=\prod_{i=1}^n\exp\left\{-\frac{(\lambda a_i)^2}{2}\int_{t_{i-1}}^{t_i}e^{2U_t}R^{(\infty,r)}_t(1-R^{(\infty,r)}_t)[\sigma^1(1-R^{(\infty,r)}_t)+\sigma^2R^{(\infty,r)}_t]dt\right\},
	\end{align}
	where $\sigma^i=2c^{(i)}+\int_{(0,\infty)}w^2m^{(i)}(dw)$ for $i=1,2$. Hence, \eqref{char_5} implies that the random vector $(Y^{(\infty)}_{t_1},Y^{(\infty)}_{t_2}-Y^{(\infty)}_{t_1},\dots,Y^{(\infty)}_{t_n}-Y^{(\infty)}_{t_{n-1}})$ has a Gaussian distribution, and hence $Y^{(\infty)}$ is a Gaussian process.
	
	We will now compute the covariance function of the process $Y^{(\infty)}$. By \eqref{char_5} we obtain for $t_1<t_2$
	\begin{align*}
		\E\Big[&e^{i(b_1Y^{(\infty)}_{t_1}+b_2Y^{(\infty)}_{t_2})}\Big]\notag\\&=\exp\Bigg\{-\sum_{i=1}^2\frac{b_i^2}{2}\int_{0}^{t_i}e^{2U_t}R^{(\infty,r)}_t(1-R^{(\infty,r)}_t)[\sigma^1(1-R^{(\infty,r)}_t)+\sigma^2R^{(\infty,r)}_t]dt\\
		&-b_1b_2\int_{0}^{t_1}e^{2U_t}R^{(\infty,r)}_t(1-R^{(\infty,r)}_t)[\sigma^1(1-R^{(\infty,r)}_t)+\sigma^2R^{(\infty,r)}_t]dt\Bigg\}, \qquad b_1,b_2\in\R.
	\end{align*}
	Therefore, the covariance function of the process $Y^{(\infty)}$ is given for, $s,t\geq0$, by
	\begin{align*}
		C_{Y^{(\infty)}}(s,t)=\int_0^{s\wedge t}e^{2U_u}R^{(\infty,r)}_u(1-R^{(\infty,r)}_u)[\sigma^1(1-R^{(\infty,r)}_u)+\sigma^2R^{(\infty,r)}_u]du.
	\end{align*}
	\section{Proof of Lemma \ref{conv_A_2}}\label{App_5}
	Using Doob's inequality together with It\^o's isometry we obtain
	\begin{align}\label{con_mart_fluc_0}
		&\E\Bigg(\sup_{t\in[0,T]}\Bigg|\int_0^te^{U_s}\sqrt{2R^{(z,r)}_{s-}(1-R^{(z,r)}_{s-})[c^{(1)}(1-R^{(z,r)}_{s-})+c^{(2)}R^{(z,r)}_{s-}]}dB_s\notag\\&\hspace{ 3cm}-\int_0^te^{U_s}\sqrt{2R^{(\infty,r)}_s(1-R^{(\infty,r)}_s)[c^{(1)}(1-R^{(\infty,r)}_s)+c^{(2)}R^{(\infty,r)}_s]}dB_s\Bigg|^2\Bigg)\notag\\
		&\leq4 \E\Bigg[\int_0^Te^{2U_s}\Bigg(\sqrt{2R^{(z,r)}_s(1-R^{(z,r)}_s)[c^{(1)}(1-R^{(z,r)}_s)+c^{(2)}R^{(z,r)}_s]}\notag\\&\hspace{ 4cm}-\sqrt{2R^{(\infty,r)}_s(1-R^{(\infty,r)}_s)[c^{(1)}(1-R^{(\infty,r)}_s)+c^{(2)}R^{(\infty,r)}_s]}\Bigg)^2ds\Bigg]\notag\\
		&\leq4 \E\Bigg[\int_0^Te^{2U_s}\Bigg|2R^{(z,r)}_s(1-R^{(z,r)}_s)[c^{(1)}(1-R^{(z,r)}_s)+c^{(2)}R^{(z,r)}_s]\notag\\&\hspace{ 5cm}-2R^{(\infty,r)}_s(1-R^{(\infty,r)}_s)[c^{(1)}(1-R^{(\infty,r)}_s)+c^{(2)}R^{(\infty,r)}_s]\Bigg|ds\Bigg]\notag\\
		&\leq 12(c^{(1)}+c^{(2)})\int_0^Te^{2U_s}\E\left[|R^{(z,r)}_s-R^{(\infty,r)}_s|\right]ds\leq 12(c^{(1)}+c^{(2)})C(T)\E\left[\sup_{t\in[0,T]}|R^{(z,r)}_t-R^{(\infty,r)}_t|^2\right]^{1/2},
	\end{align}
	where $C(T)=\int_0^Te^{2U_s}ds$.
	In a similar way, by Doob's inequality 
	\begin{align}\label{con_mart_fluc_1}
		&\E\Bigg[\sup_{t\in[0,T]}\Bigg|\int_0^t\int_{(0,\infty)^2}\sqrt{z}e^{U_s}g^{(z)}(R^{(z,r)}_{s-},w,v)\tilde{N}_1(ds,dw,dv)\notag\\&\hspace{ 5cm}-\int_0^t\int_{(0,\infty)^2}\sqrt{z}e^{U_s}g^{(z)}(R^{(\infty,r)}_s,w,v)\tilde{N}_1(ds,dw,dv)\Bigg|^2\Bigg]\notag\\
		&\leq4 \E\Bigg[\int_0^T\int_{(0,\infty)}e^{2U_s}|R^{(\infty,r)}_s-R^{(z,r)}_s|^2\frac{w^2z^2}{(z+w)^2}m^{(1)}(dw)ds\Bigg]\notag\\&\hspace{ 5cm}+4\E\Bigg[\int_0^T\int_{(0,\infty)}e^{2U_s}|R^{(\infty,r)}_s-R^{(z,r)}_s|\frac{w^2z^2}{(z+w)^2}m^{(1)}(dw)ds\Bigg]\notag\\
		&\leq 8C(T)\int_{(0,\infty)}w^2m^{(1)}(dw)\E\left[\sup_{t\in[0,T]}|R^{(z,r)}_t-R^{(\infty,r)}_t|^2\right]^{1/2}.
	\end{align}
	Proceeding as in \eqref{con_mart_fluc_1}, we obtain
	\begin{align}\label{con_mart_fluc_2}
		&\E\Bigg[\sup_{t\in[0,T]}\Bigg|\int_0^t\int_{(0,\infty)^2}e^{U_s}\sqrt{z}h^{(z)}(R^{(z,r)}_{s-},w,v)\tilde{N}_2(ds,dw,dv)\notag\\&\hspace{ 3cm}-\int_0^t\int_{(0,\infty)^2}\sqrt{z}e^{U_s}h^{(z)}(R^{(\infty,r)}_{s},w,v)\tilde{N}_2(ds,dw,dv)\Bigg|^2\Bigg]\notag\\
		&\leq 8C(T)\int_{(0,\infty)}w^2m^{(2)}(dw)\E\left[\sup_{t\in[0,T]}|R^{(z,r)}_t-R^{(\infty,r)}_t|^2\right]^{1/2}.
	\end{align}
	By \eqref{est_mart_discont_3}, we obtain that
	\begin{align}\label{con_mart_fluc_3}
		\E\Bigg[\Bigg(\sup_{t\in[0,T]}&\int_0^t\int_{(0,\infty)}\sqrt{z}e^{U_s}\tilde{g}^{(z)}(R^{(z,r)}_{s-},w)\tilde{N}_3(ds,dw)\Bigg)^2\Bigg]\notag\\&+\E\Bigg[\Bigg(\sup_{t\in[0,T]}\int_0^t\int_{(0,\infty)}\sqrt{z}e^{U_s}\tilde{h}^{(z)}(R^{(z,r)}_{s-},w)\tilde{N}_4(ds,dw)\Bigg)^2\Bigg]\notag\\
		&\leq C(T)\left(\int_{(0,\infty)}\frac{w^2z}{(w+z)^2}\nu^{(1)}(dw)+\int_{(0,\infty)}\frac{w^2z}{(w+z)^2}\nu^{(2)}(dw)\right). 
	\end{align}
	Hence, by using \eqref{con_mart_fluc_3} and the fact that $\int_{(0,\infty)}w\nu^{(i)}(dw)<\infty$ for $i=1,2$, we can find a constant $C(T,z)$ such that
	\begin{align}\label{con_mart_fluc_4}
		\E\Bigg[\Bigg(\sup_{t\in[0,T]}\int_0^t&\int_{(0,\infty)}\sqrt{z}e^{U_s}\tilde{g}^{(z)}(R^{(z,r)}_{s-},w)\tilde{N}_3(ds,dw)\Bigg)^2\Bigg]\notag\\&+\E\Bigg[\Bigg(\sup_{t\in[0,T]}\int_0^t\int_{(0,\infty)}\sqrt{z}e^{U_s}\tilde{h}^{(z)}(R^{(z,r)}_{s-},w)\tilde{N}_4(ds,dw)\Bigg)^2\Bigg]\leq C(T,z),
	\end{align}
	and such that $\lim_{z\to\infty}C(z,T)=0$.
	
	Finally, using \eqref{con_mart_fluc_0}, \eqref{con_mart_fluc_1}, \eqref{con_mart_fluc_2}, and \eqref{con_mart_fluc_4} gives
	\begin{align*}
		\E\left[\sup_{t\in[0,T]}|X^{(z)}_t-\sqrt{z}A^{(4,z)}_t|^2\right]\leq C(z,T)+K(T)\E\left[\sup_{t\in[0,T]}|R^{(z,r)}_t-R^{(\infty,r)}_t|^2\right]^{1/2}.
	\end{align*}
	The result now follows from Theorem \ref{large_pop}.
	\section{Proof of Theorem \ref{theo_dual}}\label{App_6}
	We will consider $\N_0\cup \{\Delta\}$ endowed with the discrete topology and $\N_0\cup \{\Delta\}\times [0,1]$ with the product topology. We recall that for every fixed $r\in[0,1]$, $H(n,r)=r^n$ with $H(0,r)=1$  and $H(\Delta,r)=0$, which are bounded and continuous. 
	In addition, for every fixed $k\in \N_0\cup \{\Delta\}$, $H(k,r)=r^k$ is continuous. Therefore, we conclude that $H:\N_0\cup \{\Delta\}\times [0,1]\mapsto [0,1]$ is continuous.
	
	We observe that $H(\cdot,n)$ is a polynomial in $[0,1]$ for fixed $n\in\mathbb{N}_0\cup\{\Delta\}$. This fact clearly implies that $H(\cdot,n)\in\mathcal{C}^2([0,1])$ and hence it lies in the domain of the generator $\mathcal{L}^{(z)}$. Therefore, the process
	\[
	H(n,R^{(z,r)}_t)-\int_0^t\mathcal{L}^{(z)}H(n,R^{(z,r)}_s)ds
	\]
	is a martingale.
	
	Additionally, as in the proof of Lemma 2 in \cite{GPP}, we have that for fixed $r\in[0,1]$ the function $H(\cdot,r)$ lies in the domain of the generator $\mathcal{Q}^{(z)}$, which implies that the process
	\[
	H(N^{(z,n)}_t,r)-\int_0^t\mathcal{Q}^{(z)}H(N^{(z,n)}_s,r)ds
	\]
	is also a martingale.
	In view of Proposition \ref{prop:provingduality}, we will compute $\mathcal{L}^{(z)}H(n,r)$ for $r\in[0,1]$ and $n\in\mathbb{N}$, hence using \eqref{inf_gen_cropped}
	\begin{align}\label{inf_gen_cropped_dual}
		&\mathcal{L}^{(z)}H(n,r)=nr^{n-1}\left[r(1-r)(b^{(2)}-b^{(1)})+\frac{2r(1-r)}{z}\left(c^{(2)}-c^{(1)}\right)\right]\notag\\&+n(n-1)r^{n-1}\frac{(1-r)}{z}(c^{(1)}(1-r)+c^{(2)} r)\notag\\
		&+\frac{\eta^{(1)}}{z}nr^{n-1}(1-r)+\int_{(0,1)}\left[\left(r(1-u)+u\right)^n-r^n\right]\mathbf{T^{(z)}}(\nu^{(1)})(du)\notag\\
		&+zr\int_{(0,1)}\Bigg[\left(r(1-u)+u\right)^n-r^n-\frac{u}{1-u}nr^{n-1}(1-r)1_{(0,1/(1+z))(u)}\Bigg]\mathbf{T^{(z)}}(m^{(1)})(du)\notag\\
		&+z(1-r)\int_{(0,1)}\Bigg[\left(r(1-u)\right)^n-r^n+\frac{u}{1-u}nr^{n}1_{(0,1/(1+z))(u)}\Bigg]\mathbf{T^{(z)}}(m^{(2)})(du)\notag\\
		&-\frac{\eta^{(2)}}{z}nr^{n}+\int_{(0,1)}\left[\left(r(1-u)\right)^n-r^n\right]\mathbf{T^{(z)}}(\nu^{(2)})(du).
	\end{align}
	
	(i) We note that for $r,u\in[0,1]$ and $n\in\N$
	\begin{align*}
		&(r(1-u)+u)^n-r^n-nr^{n-1}u(1-r)=\sum_{k=0}^n\binom{n}{k}r^{n-k}(1-u)^{n-k}u^k-r^n-u(1-r)nr^{n-1}\\
		&=\sum_{k=2}^n\binom{n}{k}(1-u)^{n-k}u^k(r^{n-k}-r^n)-nr^{n-1}u(1-r)(1-(1-u)^{n-1})\\
		&=\sum_{k=2}^n\binom{n}{k}(1-u)^{n-k}u^k(r^{n-k}-r^n)-n\sum_{k=1}^{n-1}\binom{n-1}{k}(1-u)^{n-1-k}u^{k+1}(r^{n-1}-r^n).
	\end{align*}
	Similarly, for $r,u\in[0,1]$ and $n\in\N$
	\begin{align*}
		(r(1-u)+u)^n-r^n&=\sum_{k=1}^n\binom{n}{k}r^{n-k}(1-u)^{n-k}u^k-r^n(1-(1-u)^n)\notag\\&=\sum_{k=1}^n\binom{n}{k}(1-u)^{n-k}u^k(r^{n-k}-r^n).
	\end{align*}
	Therefore, for $r,u\in[0,1]$ and $n\in\N$
	\begin{align}\label{gen_dual_1}
		&zr\int_{(0,1)}\Bigg[\left(r(1-u)+u\right)^n-r^n-\frac{u}{1-u}nr^{n-1}(1-r)1_{(0,1/(1+z))}(u)\Bigg]\mathbf{T^{(z)}}(m^{(1)})(du)\notag\\&=zr\int_{(0,1/(1+z))}\Bigg[\left(r(1-u)+u\right)^n-r^n-unr^{n-1}(1-r)\Bigg]\mathbf{T^{(z)}}(m^{(1)})(du)\notag\\
		&-znr^{n}(1-r)\int_{(0,1/(1+z))}\frac{u^2}{1-u}\mathbf{T^{(z)}}(m^{(1)})(du)\notag\\&+zr\int_{[1/(1+z),1)}\Bigg[\left(r(1-u)+u\right)^n-r^n\Bigg]\mathbf{T^{(z)}}(m^{(1)})(du)\notag\\
		&=\sum_{k=2}^n\binom{n}{k}z\lambda_{n,k}^{(1)}\left(\frac{1}{1+z}\right)(r^{n+1-k}-r^n)\notag\\&+\sum_{k=1}^n\binom{n}{k}z\left[\lambda_{n,k}^{(1)}(1)-\lambda_{n,k}^{(1)}\left(\frac{1}{1+z}\right)\right](r^{n+1-k}-r^n)\notag\\&+(r^{n+1}-r^n)\Bigg[\sum_{k=1}^{n-1}n\binom{n-1}{k}z\lambda_{n,k+1}^{(1)}\left(\frac{1}{1+z}\right)+zn\int_{(0,1/(1+z))}\frac{u^2}{1-u}\mathbf{T^{(z)}}(m^{(1)})(du)\notag\\
		&-\sum_{k=2}^n\binom{n}{k}z\lambda_{n,k}^{(1)}\left(\frac{1}{1+z}\right)-\sum_{k=1}^n\binom{n}{k}z\left[\lambda_{n,k}^{(1)}(1)-\lambda_{n,k}^{(1)}\left(\frac{1}{1+z}\right)\right]\Bigg]\notag\\
		&=\sum_{k=2}^n\binom{n}{k}z\lambda_{n,k}^{(1)}(1)(r^{n+1-k}-r^n)+(r^{n+1}-r^n)\Bigg[\sum_{k=2}^{n}\binom{n}{k}z\left[k\lambda_{n,k}^{(1)}\left(\frac{1}{1+z}\right)-\lambda_{n,k}^{(1)}\left(1\right)\right]\notag\\
		&+zn\int_{(0,1/(1+z))}\frac{u^2}{1-u}\mathbf{T^{(z)}}(m^{(1)})(du)-nz\left[\lambda_{n,1}^{(1)}(1)-\lambda_{n,1}^{(1)}\left(\frac{1}{1+z}\right)\right]\Bigg].
	\end{align}
	(ii)  For $r,u\in[0,1]$ and $n\in\N$, we observe
	\begin{align*}
		(r(1-u))^n-r^n+nur^n&=nur^n(1-(1-u)^{n-1})-r^n(1-(1-u)^n-nu(1-u)^{n-1})\\
		&=r^n\left(n\sum_{k=1}^{n-1}\binom{n-1}{k}(1-u)^{n-1-k}u^{k+1}-\sum_{k=2}^{n}\binom{n}{k}(1-u)^{n-k}u^k\right).
	\end{align*}
	Therefore, for $r,u\in[0,1]$ and $n\in\N$
	\begin{align}\label{gen_dual_2}
		&z(1-r)\int_{(0,1)}\Bigg[\left(r(1-u)^n\right)-r^n+\frac{u}{1-u}nr^n1_{(0,1/(1+z))}(u)\Bigg]\mathbf{T^{(z)}}(m^{(2)})(du)\notag\\
		&=z(1-r)\int_{(0,1/(1+z))}\left[\left(r(1-u)\right)^n-r^n+unr^{n}\right]\mathbf{T^{(z)}}(m^{(2)})(du)\notag\\&+znr^{n}(1-r)\int_{(0,1/(1+z))}\frac{u^2}{1-u}\mathbf{T^{(z)}}(m^{(2)})(du)\notag\\
		&+z(1-r)\int_{[1/(1+z),1)}\Bigg[\left(r(1-u)\right)^n-r^n\Bigg]\mathbf{T^{(z)}}(m^{(2)})(du)\notag\\
		&=(r^{n+1}-r^n)\Bigg[\sum_{k=2}^n\binom{n}{k}z\lambda_{n,k}^{(2)}\left(\frac{1}{1+z}\right)-\sum_{k=1}^{n-1}n\binom{n-1}{k}z\lambda_{n,k+1}^{(2)}\left(\frac{1}{1+z}\right)\notag\\&-zn\int_{(0,1/(1+z))}\frac{u^2}{1-u}\mathbf{T^{(z)}}(m^{(2)})(du)\notag\\
		&+\sum_{k=1}^n\binom{n}{k}z\left[\lambda_{n,k}^{(2)}(1)-\lambda_{n,k}^{(2)}\left(\frac{1}{1+z}\right)\right]\Bigg]\notag\\
		&=(r^{n+1}-r^n)\Bigg[\sum_{k=2}^n\binom{n}{k}z\left[\lambda_{n,k}^{(2)}(1)-k\lambda_{n,k}^{(2)}\left(\frac{1}{1+z}\right)\right]-zn\int_{(0,1/(1+z))}\frac{u^2}{1-u}\mathbf{T^{(z)}}(m^{(2)})(du)\notag\\&+nz\left[\lambda_{n,1}^{(2)}(1)-\lambda_{n,1}^{(2)}\left(\frac{1}{1+z}\right)\right]\Bigg].
	\end{align}
	(iii) For the jump terms due to immigration in the expression for $\mathcal{L}^{(z)}H(n,x)$ given in \eqref{inf_gen_cropped_dual}, we obtain for $r,u\in[0,1]$ and $n\in\N$
	\begin{align*}
		(r(1-u)+u)^n-r^n&=\sum_{k=1}^n\binom{n}{k}r^{n-k}(1-u)^{n-k}u^k-r^n(1-(1-u)^n)\notag\\&=\sum_{k=1}^n\binom{n}{k}(1-u)^{n-k}u^k(r^{n-k}-r^n).
	\end{align*}
	Hence, for $r,u\in[0,1]$ and $n\in\N$
	\begin{align}\label{gen_dual_3}
		\int_{(0,1)}\left[(r(1-u)+u)^n-r^n\right]\mathbf{T^{(z)}}(\nu^{(1)})(du)=\sum_{k=1}^n\binom{n}{k}\mu_{n,k}^1(r^{n-k}-r^n).
	\end{align}
	Similar computations give for $r,u\in[0,1]$ and $n\in\N$
	\begin{align}\label{gen_dual_4}
		\int_{(0,1)}\left[(r(1-u))^n-r^n\right]\mathbf{T^{(z)}}(\nu^{(2)})(du)=-r^n\sum_{k=1}^n\binom{n}{k}\mu_{n,k}^{(2)}.
	\end{align}
	(iv) Finally, for the terms due to the continuous part of the process $R^{(z,r)}$ in the expression for $\mathcal{L}^{(z)}H(n,x)$ given in \eqref{inf_gen_cropped_dual},  we obtain for $r,u\in[0,1]$ and $n\in\N$
	\begin{align}\label{gen_dual_5}
		nr^{n-1}\left[r(1-r)(b^{(2)}-b^{(1)})+\frac{2r(1-r)}{z}\left(c^{(2)}-c^{(1)}\right)\right]&=n(r^{n+1}-r^n)\left[(b^{(1)}-b^{(2)})+\frac{2}{z}(c^{(1)}-c^{(2)})\right],
	\end{align}
	and for $r,u\in[0,1]$ and $n\geq 2$
	\begin{align}\label{gen_dual_6}
		n(n-1)r^{n-1}\frac{(1-r)}{z}(c^{(1)}(1-r)+c^{(2)} r)&=n(n-1)(r^{n+1}-r^n)\frac{(c^{(1)}-c^{(2)})}{z}\notag\\&+n(n-1)\frac{c^{(1)}}{z}(r^{n-1}-r^n).
	\end{align}
	(v) So, putting the pieces together, we obtain using identities \eqref{gen_dual_1}-\eqref{gen_dual_6} in \eqref{inf_gen_cropped_dual}
	\begin{align*}
		&\mathcal{L}^{(z)}H(n,r)=\sum_{k=3}^n\binom{n}{k}z\lambda_{n,k}^{(1)}(1)(r^{n+1-k}-r^n)
		+\sum_{k=1}^n\binom{n}{k}\mu_{n,k}^1(r^{n-k}-r^n)\notag\\&-r^n\left[n\frac{\eta^{(2)}}{z}+\sum_{k=1}^n\binom{n}{k}\mu_{n,k}^{(2)}\right]+(r^{n-1}-r^n)\left[n(n-1)\frac{c^{(1)}}{z}+\binom{n}{2}z\lambda_{n,2}^{(1)}(1)+n\frac{\eta^{(1)}}{z}\right]\notag\\&+(r^{n+1}-r^n)\Bigg[n(n-1)\frac{(c^{(1)}-c^{(2)})}{z}+\sum_{k=2}^{n}\binom{n}{k}z\left[k\lambda_{n,k}^{(1)}\left(\frac{1}{1+z}\right)-\lambda_{n,k}^{(1)}\left(1\right)\right]\notag\\&+zn\int_{(0,1/(1+z))}\frac{u^2}{1-u}\mathbf{T^{(z)}}(m^{(1)})(du)-nz\left[\lambda_{n,1}^{(1)}(1)-\lambda_{n,1}^{(1)}\left(\frac{1}{1+z}\right)\right]\notag\\&+\sum_{k=2}^n\binom{n}{k}z\left[\lambda_{n,k}^{(2)}(1)-k\lambda_{n,k}^{(2)}\left(\frac{1}{1+z}\right)\right]-zn\int_{(0,1/(1+z))}\frac{u^2}{1-u}\mathbf{T^{(z)}}(m^{(2)})(du)\notag\\&+nz\left[\lambda_{n,1}^{(2)}(1)-\lambda_{n,1}^{(2)}\left(\frac{1}{1+z}\right)\right]+n\left((b^{(1)}-b^{(2)})+\frac{2}{z}(c^{(1)}-c^{(2)})\right)\Bigg].
	\end{align*}
	Further computations give for $r\in[0,1]$ and $n\in\mathbb{N}$
	\begin{align*}
		&\mathcal{L}^{(z)}H(n,r)=\sum_{k=3}^n\left[\binom{n}{k}z\lambda_{n,k}^{(1)}(1)+\binom{n}{k-1}\mu_{n,k-1}^{(1)}\right](r^{n+1-k}-r^n)+\mu^{(1)}_{n,n}(1-r^n)\notag\\
		&+n(r^{n+1}-r^n)\left[(b^{(1)}-b^{(2)})+\frac{2}{z}(c^{(1)}-c^{(2)})\right]-r^n\left[n\frac{\eta^{(2)}}{z}+\sum_{k=1}^n\binom{n}{k}\mu_{n,k}^{(2)}\right]\notag
		\end{align*}
		\begin{align*}
		&+(r^{n-1}-r^n)\left[n(n-1)\frac{c^{(1)}}{z}+\binom{n}{2}z\lambda_{n,2}^{(1)}(1)+n\frac{\eta^{(1)}}{z}+n\mu_{n,1}^{(1)}\right]\notag\\&+(r^{n+1}-r^n)\Bigg\{n(n-1)\frac{(c^{(1)}-c^{(2)})}{z}
		+\sum_{k=2}^{n}\binom{n}{k}z\Bigg[k\left(\lambda_{n,k}^{(1)}\left(\frac{1}{1+z}\right)-\lambda_{n,k}^{(2)}\left(\frac{1}{1+z}\right)\right)\notag\\&-\left(\lambda_{n,k}^{(1)}\left(1\right)-\lambda_{n,k}^{(2)}\left(1\right)\right)\Bigg]-nz\left[\left(\lambda_{n,1}^{(1)}\left(1\right)-\lambda_{n,1}^{(2)}\left(1\right)\right)-\left(\lambda_{n,1}^{(1)}\left(\frac{1}{1+z}\right)-\lambda_{n,1}^{(2)}\left(\frac{1}{1+z}\right)\right)\right]\\
		&+zn\int_{(0,1/(1+z))}\frac{u^2}{1-u}\mathbf{T^{(z)}}(m^{(1)})(du)-zn\int_{(0,1/(1+z))}\frac{u^2}{1-u}\mathbf{T^{(z)}}(m^{(2)})(du)\Bigg\}\notag\\
		&=\sum_{k=2}^n\left[\binom{n}{k}\overline{\lambda}_{n,k}^{(1)}+\binom{n}{k-1}\overline{\mu}_{n,k-1}^{(1)}\right](r^{n+1-k}-r^n)+\overline{\mu}^{(1)}_{n,n}(1-r^n)
		-\alpha_nr^n\notag\\
		&
		+(r^{n+1}-r^n)\left[
		ns+\sum_{k=2}^{n}\binom{n}{k}\kappa_k+\beta_n\right]
		=\mathcal{Q}^{(z)}H(n,r).
	\end{align*}
	Meanwhile, for the case $n=0$, we have that $\mathcal{L}^{(z)}H(0,r)=0=\mathcal{Q}^{(z)}H(0,r)$ for $r\in[0,1]$. Similarly, we have that $\mathcal{L}^{(z)}H(\Delta,r)=0=\mathcal{Q}^{(z)}H(\Delta,r)$ for $r\in[0,1]$. Therefore, the result follows from Proposition \ref{prop:provingduality}.
\end{appendix}

\end{document}